\documentclass[12pt]{article}
\usepackage{amsfonts, amsmath, amssymb, latexsym, eucal, amscd} 

\usepackage{amsthm}
\usepackage{amscd}

\usepackage{cite}

\newtheorem{definition}{Definition}[section]

\newtheorem{lemma}[definition]{Lemma}

\newtheorem{remark}[definition]{Remark}
\newtheorem{example}[definition]{Example}
\newtheorem{conjecture}[definition]{Conjecture}

\newtheorem{proposition}[definition]{Proposition}

\begin{document} 

\title{\bf Circular Hessenberg pairs and \\ the
tridiagonal relations
}
\author{
Kazumasa Nomura and Paul Terwilliger}
\date{}

\maketitle
\begin{abstract}  A square matrix is said to be Hessenberg whenever each entry below the subdiagonal is zero, and each entry on the subdiagonal is nonzero.
A Hessenberg matrix is called circular whenever the top-right corner entry is nonzero, and every other entry above the superdiagonal is zero.
A circular Hessenberg pair consists of two diagonalizable linear maps on a nonzero finite-dimensional vector space, that each act on an eigenbasis of the other one
in a circular Hessenberg fashion. In 2022, Jae-ho Lee conjectured that a circular Hessenberg pair satisfies two relations called the tridiagonal relations.
In the present paper, we prove Lee's conjecture. Our proof is not elementary.
\medskip

\noindent
{\bf Keywords}. Leonard pair, Hessenberg matrix, Hessenberg pair, tridiagonal relations.
\hfil\break
\noindent {\bf 2020 Mathematics Subject Classification}.
Primary: 05E30. Secondary: 15A04, 15A21.
 \end{abstract}
 
 \section{Introduction}
 \noindent There is a mathematical object called a Leonard pair \cite{LS99}, that appears in algebraic combinatorics \cite{bbit,int},  the theory of orthogonal polynomials \cite{qrac},
  and representation theory \cite{LSintro}.
 A Leonard pair consists of two diagonalizable linear maps on
 a nonzero finite-dimensional vector space, that each act on an eigenbasis of the other one in an irreducible tridiagonal fashion  \cite[Definition~1.1]{LS99}. The original application
 of the Leonard pair concept was twofold: (i)  to illuminate a theorem of Leonard \cite[p.~260]{bannai}, \cite{leonard} concerning the orthogonal polynomials that make up the terminating branch of the Askey scheme;
 (ii) to describe the irreducible modules for the subconstituent algebra of a $Q$-polynomial distance-regular graph \cite{tSub1, tSub2, tSub3}.
In the study of a Leonard pair, it is natural to consider a related object called a Leonard system \cite[Definition~1.4]{LS99}. In \cite[Theorem~1.9]{LS99} the Leonard systems are classified up to isomorphism.
A key step in the classification is the following fact. 
By \cite[Theorem~1.12]{LS99}, a Leonard pair satisfies two relations called the tridiagonal relations; these are shown in \eqref{eq:CTD1}, \eqref{eq:CTD2} below.
More information about Leonard pairs and systems can be found in \cite{nomKraw, nomSpinModel, LS24,aa, LSnotes, vidunas}.
 \medskip

\noindent  In \cite{godjali}, Ali Godjali introduced a generalization of a Leonard pair called a Hessenberg pair. In \cite[Definition~2.2]{godjali} Godjali defined the concept of a Hessenberg system,
and in \cite[Theorem~6.3]{godjali} he classfied up to isomorphism the Hessenberg systems. In \cite{godjali2}, Godjali used Hessenberg pairs and systems to investigate double Vandermonde matrices.
\medskip

\noindent In \cite{JHL}, Jae-ho Lee introduced a family of Hessenberg pairs and Hessenberg systems, said to be circular. In \cite[Conjecture~2.13]{JHL}, Lee conjectured that
every circular Hessenberg pair satisfies the tridiagonal relations. In \cite[Theorem~5.6]{JHL}, Lee classified up to isomorphism the circular Hessenberg systems
for which the corresponding circular Hessenberg pair satisfies  the tridiagonal relations. 
\medskip

\noindent  In this paper, we prove Lee's conjecture. Our proof is not elementary.
\medskip

\noindent The paper is organized as follows. Section 2 contains some preliminaries.
In Section 3, we recall the concepts of a Hessenberg pair and  system.
In Section 4, we recall the concepts of a circular Hessenberg pair and system.
In Section 5, we discuss the tridiagonal relations and Lee's conjecture. We also prove Lee's conjecture for $d=2$, where $d+1$ is the dimension of the  vector space mentioned in the conjecture.
In Section 6, we prove Lee's conjecture for $d=3$.
In Section 7, we prove Lee's conjecture for $d\geq 4$.

 \section{Preliminaries}  In this section, we review some concepts and notation that will be used throughout the paper.
 Recall the natural numbers $\mathbb N = \lbrace 0,1,2,\ldots\rbrace$ and integers $\mathbb Z = \lbrace 0,\pm1, \pm2, \ldots\rbrace$.
 For $n \in \mathbb N$ a sequence $x_0, x_1,\ldots, x_n$ is often denoted by $\lbrace x_i \rbrace_{i=0}^n$.
  Let $\mathbb F$ denote a field.  An element of $\mathbb F$ is called a {\it scalar}.
  Every vector space and algebra mentioned in this paper, is understood to be over $\mathbb F$.
  Fix an integer $d\geq 2$, and let $V$ denote a vector space 
with dimension $d+1$. Let ${\rm End}(V)$ denote the algebra consisting of the $\mathbb F$-linear maps from $V$ to $V$.
Let ${\rm Mat}_{d+1}(\mathbb F)$ denote the algebra of $(d+1) \times (d+1)$ matrices that have all entries in $\mathbb F$.
We index the rows and columns by $0,1,\ldots, d$.
We recall how each basis for $V$ gives an algebra isomorphism 
${\rm End}(V) \to {\rm Mat}_{d+1}(\mathbb F)$. Let $\lbrace v_i \rbrace_{i=0}^d$ denote a basis for $V$, and let $A \in {\rm End}(V)$.
By the {\it matrix that represents $A$ with respect to $\lbrace v_i \rbrace_{i=0}^d$}, we mean the matrix $B \in  {\rm Mat}_{d+1}(\mathbb F)$
such that $A v_j = \sum_{i=0}^d B_{i,j} v_i$ for $0 \leq j \leq d$. The isomorphism sends $A$ to the matrix in ${\rm Mat}_{d+1}(\mathbb F)$ that represents $A$ with
respect to  $\lbrace v_i \rbrace_{i=0}^d$. 
Let $A \in {\rm End}(V)$. We say that $A$ is {\it diagonalizable} whenever $V$ is spanned by the eigenspaces of $A$.
We say that $A$ is {\it multiplicity-free} whenever $A$ is diagonalizable and each eigenspace of $A$ has dimension one. Assume that $A$ is multiplicity-free,
and let $\lbrace V_i \rbrace_{i=0}^d$ denote an ordering of the eigenspaces of $A$. By linear algebra, the sum $V=\sum_{i=0}^d V_i$ is direct.
For $0 \leq i \leq d$ let $\theta_i$ denote the eigenvalue
of $A$ for $V_i$. The scalars $\lbrace \theta_i \rbrace_{i=0}^d$ are mutually distinct.
For $0 \leq i \leq d$ define $E_i \in {\rm End}(V)$ such that $(E_i-I)V_i=0$ and $E_i V_j =0$ if $j \not=i$ $(0 \leq j \leq d)$. We call $E_i$
the {\it primitive idempotent} of $A$ associated with $V_i$ (or $\theta_i$). We have
(i) $E_i E_j = \delta_{i,j} E_i$ $(0 \leq i,j\leq d)$; 
(ii) $I = \sum_{i=0}^d E_i$;
(iii) $A = \sum_{i=0}^d \theta_i E_i$;
(iv) $A E_i = \theta_i E_i = E_i A$ $(0 \leq i \leq d)$;
(v) $V_i = E_i V$ $(0 \leq i \leq d)$. 
Moreover
\begin{align} \label{eq;Eformula}
E_i = \prod_{\stackrel{0 \leq j \leq d}{j \neq i}}
       \frac{A-\theta_jI}{\theta_i-\theta_j}
\qquad \qquad (0 \leq i \leq d).
\end{align}
By linear algebra,  both $\lbrace E_i \rbrace_{i=0}^d$ and $\lbrace A^i \rbrace_{i=0}^d$ are bases for the subalgebra of ${\rm End}(V)$ generated by $A$.
 \medskip

 \noindent For $R, S \in {\rm End}(V)$ their commutator is $\lbrack R, S\rbrack=RS-SR$.
 \medskip
 
 \noindent For an indeterminate $\lambda$, let  $\mathbb F \lbrack \lambda \rbrack$ denote the algebra consisting of the polynomials in $\lambda$
 that have all coefficients in $\mathbb F$.
For commuting indeterminates 
  $\lambda, \mu $ let $\mathbb F \lbrack \lambda, \mu\rbrack$ denote the algebra consisting of the polynomials in $\lambda, \mu$
 that have all coefficients in $\mathbb F$. 
 \medskip
 
 \noindent For $B \in {\rm Mat}_{d+1}(\mathbb F)$ let ${\rm det}(B)$ and ${\rm tr}(B)$ denote the determinant and trace of $B$, respectively.
 
 \section{Hessenberg pairs and systems}
 
 \noindent In this section, we recall the concept of a Hessenberg pair and a Hessenberg system.
 \medskip
 
 \begin{definition} \label{def:HM} \rm
 A matrix $M \in {\rm Mat}_{d+1} (\mathbb F)$ is called {\it Hessenberg} whenever 
 \begin{align*}
 M_{i,j}=  \begin{cases} 0 &\hbox{ \rm if  $i-j>1$},\\
                          \not=0 &\hbox{ \rm if  $i-j=1$}
                      \end{cases} \qquad \qquad (0 \leq i,j\leq d).
 \end{align*}
 \end{definition}
 
 \begin{example} \rm A Hessenberg matrix with $d=5$ looks as follows:
 \begin{align*}
 \text{Hessenberg:} \qquad
\begin{pmatrix}
* & *  & * & * & * & *  \\
\bullet & * & * & * & * & *  \\
 0 & \bullet &  * & * & * & *  \\
0 &  0 & \bullet &  * & * & * \\
0 & 0 &  0 & \bullet &  * & * \\
0 & 0 & 0 &  0 & \bullet &  *  \\
\end{pmatrix}.
\end{align*}
Here
 $\bullet$ denotes an entry that must be nonzero,
and $*$ denotes an entry that might be nonzero.
 \end{example}
 
 \begin{definition} \label{def:HP} \rm (See \cite[Definition~1.1]{godjali}.)
 A {\it Hessenberg pair on $V$} is an ordered pair $A, A^*$ of elements in ${\rm End}(V)$ such that:
 \begin{enumerate}
 \item[\rm (i)]  there exists a basis for $V$ with respect to which the matrix representing $A$ is diagonal and the matrix representing $A^*$ is Hessenberg;
 \item[\rm (ii)] there exists a basis for $V$ with respect to which the matrix representing $A^*$ is diagonal and the matrix representing $A$ is Hessenberg.
 \end{enumerate}
 \end{definition}
 
 \begin{remark} \rm Our definition of a Hessenberg pair is slightly different from the one in \cite{godjali}. What we call a Hessenberg pair is called
 a thin Hessenberg pair in \cite{godjali}.
 \end{remark}
 
  \noindent We mention one feature of a Hessenberg pair.
 \begin{lemma} \label{lem:MF} {\rm (See \cite[Lemma~2.1]{godjali}.)}
 Let $A, A^*$ denote a Hessenberg pair on $V$. Then each of $A, A^*$ is multiplicity-free.
 \end{lemma}
 
 \noindent  Next, we characterize the Hessenberg pairs on $V$. We will use the following  concepts.
 Let  $M \in {\rm Mat}_{d+1}(\mathbb F)$. Then $M$ is called {\it lower bidiagonal} whenever $M$ is lower triangular and Hessenberg.
 The matrix $M$ is called {\it upper bidiagonal} whenever the transpose $M^{\rm t}$ is lower bidiagonal.

 \begin{lemma} \label{ex:HP} {\rm (See \cite[Proposition~7.6 and Lemma~7.7]{LSnotes}.)} Elements $A, A^*$ in ${\rm End}(V)$ form a Hessenberg pair on $V$ if and only
 if there exists a basis for $V$ with respect to which the matrices representing $A$ and $A^*$ have the following forms:
 \begin{enumerate}
 \item[\rm (i)] the matrix representing $A$ is lower bidiagonal with  diagonal entries mutually distinct;
  \item[\rm (ii)] the matrix representing $A^*$ is upper bidiagonal with diagonal entries mutually distinct.
 \end{enumerate}
 \end{lemma}
 
 \noindent In the study of Hessenberg pairs, it is convenient to bring in a related object called a Hessenberg system.
 
 \begin{definition}\label{def:HS} \rm  (See \cite[Definition~2.2]{godjali}.) A {\it Hessenberg system on $V$} is a sequence
 \begin{align*}
\Phi = (A; \lbrace E_i \rbrace_{i=0}^d; A^*; \lbrace E^*_i \rbrace_{i=0}^d)
 \end{align*}
 of elements in ${\rm End}(V)$ that satisfy the conditions (i)--(v) below:
 \begin{enumerate}
 \item[\rm (i)] each of $A, A^*$ is multiplicity-free;
 \item[\rm (ii)] $\lbrace E_i \rbrace_{i=0}^d$ is an ordering of the primitive idempotents of $A$;
 \item[\rm (iii)] $\lbrace E^*_i \rbrace_{i=0}^d$ is an ordering of the primitive idempotents of $A^*$;
\item[\rm (iv)] $E_i A^* E_j =  \begin{cases} 0 &\hbox{ \rm if  $i-j >1$},\\
                          \not=0 &\hbox{ \rm if  $i-j=1$}
                      \end{cases} \qquad \qquad (0 \leq i,j\leq d)$;
\item[\rm (v)]   $E^*_i A E^*_j =  \begin{cases} 0 &\hbox{ \rm if  $i-j>1$}, \\
                          \not=0 &\hbox{ \rm if  $i-j=1$}
                      \end{cases} \qquad \qquad (0 \leq i,j\leq d)$.
\end{enumerate}
 \end{definition}
  \begin{remark} \rm  What we call a Hessenberg system is called
 a thin Hessenberg system in \cite{godjali}.
 \end{remark}

 \noindent The Hessenberg pairs and systems are related as follows. Let 
$(A; \lbrace E_i \rbrace_{i=0}^d; A^*; \lbrace E^*_i \rbrace_{i=0}^d)$ denote a Hessenberg system on $V$. Then $A, A^*$ form a Hessenberg pair on $V$.
Conversely, let $A, A^*$ denote a Hessenberg pair on $V$. Let $\lbrace v_i \rbrace_{i=0}^d$ (resp. $\lbrace v^*_i \rbrace_{i=0}^d$) 
denote a basis for $V$ that satisfies Definition \ref{def:HP}(i) (resp.   Definition \ref{def:HP}(ii)). For $0 \leq i \leq d$
the vector $v_i$ (resp. $v^*_i$) is an eigenvector for $A$ (resp. $A^*$); let $E_i$ (resp. $E^*_i$) denote the corresponding primitive idempotent of $A$ (resp. $A^*$).
 Then the sequence
$(A; \lbrace E_i \rbrace_{i=0}^d; A^*; \lbrace E^*_i \rbrace_{i=0}^d)$ is a Hessenberg system on $V$.
 \medskip
 
 \noindent Let $\Phi = (A; \lbrace E_i \rbrace_{i=0}^d; A^*; \lbrace E^*_i \rbrace_{i=0}^d)$ denote a Hessenberg system on $V$. We now recall the parameter array of $\Phi$. This parameter
 array was introduced in \cite{godjali} and investigated thoroughly in \cite{godjali2}, \cite{JHL}. For our purpose, the following definition will suffice.
 The parameter array of $\Phi$ 
 is the unique sequence of scalars
 \begin{align*}
\bigl( \lbrace \theta_i \rbrace_{i=0}^d; \lbrace \theta^*_i \rbrace_{i=0}^d; \lbrace \phi_i \rbrace_{i=1}^d\bigr)
\end{align*}
with the following features (see \cite[Definition~3.1, Proposition~5.9]{godjali} and  \cite[Proposition~7.6, Lemma~7.7]{LSnotes}).
\begin{enumerate}
\item[\rm (i)]   $\theta_i$ is the eigenvalue of $A$ for the primitive idempotent $E_i$ $(0 \leq i \leq d)$;
 \item[\rm (ii)] $\theta^*_i$ is the eigenvalue of $A^*$ for the primitive idempotent $E^*_i$ $(0 \leq i \leq d)$;
\item[\rm (iii)] there exists a basis for $V$ with respect to which  $A$ and $A^*$ look as follows:
\begin{align*}
A : \quad \begin{pmatrix}  \theta_d &&&& &{\bf 0} \\ 
                                    1 &\theta_{d-1} &&&&\\
                                       &1&\theta_{d-2} &&& \\
                                       &&\cdot&\cdot&& \\
                                       &&&\cdot &\cdot&\\
                                      {\bf 0}&&&&1&\theta_0
                                       \end{pmatrix}, \qquad \qquad 
 A^*: \quad \begin{pmatrix} \theta^*_0&\phi_1 &&&& {\bf 0} \\ 
                                     &\theta^*_1&\phi_2&&&& \\
                                      &&\theta^*_2& \cdot &&\\
                                      &&&\cdot&\cdot &\\
                                      &&&&\cdot&\phi_d\\
                                    {\bf 0}&&&&&\theta^*_d
                                    \end{pmatrix}. 
\end{align*}
\end{enumerate}

\noindent For Hessenberg systems the concept of isomorphism is defined in \cite[Definition~2.5]{godjali}.

\begin{proposition} \label{prop:PA} {\rm (See \cite[Theorem~6.3]{godjali}.)} For a sequence of scalars
\begin{align} \label{eq:PA}
\bigl( \lbrace \theta_i \rbrace_{i=0}^d; \lbrace \theta^*_i \rbrace_{i=0}^d; \lbrace \phi_i \rbrace_{i=1}^d\bigr)
\end{align}
the following are equivalent:
\begin{enumerate}
\item[\rm (i)] there exists a Hessenberg system $\Phi$ on $V$ that has  parameter array \eqref{eq:PA};
\item[\rm (ii)] $\lbrace \theta_i \rbrace_{i=0}^d$ are mutually distinct, $\lbrace \theta^*_i \rbrace_{i=0}^d$ are mutually distinct, and  $\lbrace \phi_i \rbrace_{i=1}^d$ are nonzero.
\end{enumerate}
Assume that {\rm (i), (ii)} hold. Then $\Phi$ is unique up to isomorphism of Hessenberg systems.
\end{proposition}

\noindent We have some comments.

\begin{lemma} \label{lem:dual} {\rm (See \cite[Definition~2.6 and Lemma~5.1]{godjali}.) }
Let  
\begin{align*}
\Phi = (A; \lbrace E_i \rbrace_{i=0}^d; A^*; \lbrace E^*_i \rbrace_{i=0}^d)
\end{align*}
 denote a Hessenberg system on $V$, with parameter array
 \begin{align*}
\bigl( \lbrace \theta_i \rbrace_{i=0}^d; \lbrace \theta^*_i \rbrace_{i=0}^d; \lbrace \phi_i \rbrace_{i=1}^d\bigr).
\end{align*}
Then the sequence 
\begin{align*}
\Phi^*=  (A^*; \lbrace E^*_i \rbrace_{i=0}^d; A; \lbrace E_i \rbrace_{i=0}^d)
\end{align*}
 is a Hessenberg system on $V$,
with parameter array
\begin{align*}
\bigl( \lbrace \theta^*_i \rbrace_{i=0}^d; \lbrace \theta_i \rbrace_{i=0}^d; \lbrace \phi_{d-i+1} \rbrace_{i=1}^d\bigr).
\end{align*}
\end{lemma} 

\begin{definition} \label{def:dual} \rm Referring to Lemma \ref{lem:dual}, we call $\Phi^*$ the {\it dual of $\Phi$}.
For any object $h$ attached to $\Phi$, let $h^*$ denote the corresponding object attached to $\Phi^*$.
\end{definition}

\section{Circular Hessenberg pairs and systems}

\noindent  In \cite{JHL}, Jae-ho Lee introduced a family of Hessenberg pairs and systems, said to be circular.
In this section, we describe the circular Hessenberg pairs and systems.

 \begin{definition} \label{def:HMc} \rm (See \cite[Section~2]{JHL}.)
 A Hessenberg matrix $M \in {\rm Mat}_{d+1} (\mathbb F)$ is called {\it circular} whenever 
 \begin{align*}
 M_{i,j}=  \begin{cases} 0 &\hbox{ \rm if  $1<j-i<d$},\\
                          \not=0 &\hbox{ \rm if  $j-i=d$}
                      \end{cases} \qquad \qquad (0 \leq i,j\leq d).
 \end{align*}
 \end{definition}
 
 \begin{example} \rm A circular Hessenberg matrix with $d=5$ looks as follows:
 \begin{align*}
 \text{Circular Hessenberg:} \qquad
\begin{pmatrix}
* & *  & 0 & 0& 0 & \bullet \\
\bullet & * & * & 0 & 0 & 0  \\
 0 & \bullet &  * & * & 0 & 0  \\
0 &  0 & \bullet &  * & * & 0 \\
0 & 0 &  0 & \bullet &  * & * \\
0 & 0 & 0 &  0 & \bullet &  *  \\
\end{pmatrix}.
\end{align*}
Here
 $\bullet$ denotes an entry that must be nonzero,
and $*$ denotes an entry that might be nonzero.
 \end{example}
 
 \begin{definition} \label{def:HPc} \rm (See \cite[Definition~2.9]{JHL}.)
 A {\it circular Hessenberg pair on $V$} is an ordered pair $A, A^*$ of elements in ${\rm End}(V)$ such that:
 \begin{enumerate}
 \item[\rm (i)]  there exists a basis for $V$ with respect to which the matrix representing $A$ is diagonal and the matrix representing $A^*$ is circular Hessenberg;
 \item[\rm (ii)] there exists a basis for $V$ with respect to which the matrix representing $A^*$ is diagonal and the matrix representing $A$ is circular Hessenberg.
 \end{enumerate}
 \end{definition}
 
\noindent We note that a circular Hessenberg pair on $V$ is a Hessenberg pair on $V$.

\begin{definition}\label{def:circ} \rm (See \cite[Definition~2.11]{JHL}.) 
A {\it circular Hessenberg system on $V$} is a Hessenberg system 
\begin{align*}
\Phi = (A; \lbrace E_i \rbrace_{i=0}^d; A^*; \lbrace E^*_i \rbrace_{i=0}^d)
\end{align*}
on $V$ such that:
\begin{enumerate}
\item[\rm (i)] $E_i A^* E_j =  \begin{cases} 0 &\hbox{ \rm if  $1 < j-i<d$},\\
                          \not=0 &\hbox{ \rm if  $j-i=d$}
                      \end{cases} \qquad \qquad (0 \leq i,j\leq d)$;
\item[\rm (ii)]   $E^*_i A E^*_j =  \begin{cases} 0 &\hbox{ \rm if  $1 < j-i<d$}, \\
                          \not=0 &\hbox{ \rm if  $j-i=d$}
                      \end{cases} \qquad \qquad (0 \leq i,j\leq d)$.
\end{enumerate}
\end{definition} 

\noindent We clarify how circular Hessenberg pairs and systems are related.  Let 
$(A; \lbrace E_i \rbrace_{i=0}^d; A^*; \lbrace E^*_i \rbrace_{i=0}^d)$ denote a circular Hessenberg system on $V$. Then $A, A^*$ form a circular Hessenberg pair on $V$.
Conversely, let $A, A^*$ denote a circular Hessenberg pair on $V$. Let $\lbrace v_i \rbrace_{i=0}^d$ (resp. $\lbrace v^*_i \rbrace_{i=0}^d$) 
denote a basis for $V$ that satisfies Definition \ref{def:HPc}(i) (resp.   Definition \ref{def:HPc}(ii)). For $0 \leq i \leq d$
the vector $v_i$ (resp. $v^*_i$) is an eigenvector for $A$ (resp. $A^*$); let $E_i$ (resp. $E^*_i$) denote the corresponding primitive idempotent of $A$ (resp. $A^*$).
 Then the sequence
$(A; \lbrace E_i \rbrace_{i=0}^d; A^*; \lbrace E^*_i \rbrace_{i=0}^d)$ is a circular Hessenberg system on $V$.
 \medskip

\begin{lemma} \label{lem:dualC} A Hessenberg system $\Phi$ is circular if and only if the Hessenberg system $\Phi^*$ is circular.
\end{lemma}
\begin{proof} By Definitions \ref{def:dual}, \ref{def:circ}.
\end{proof}

\section{The tridiagonal relations} 

In \cite{JHL}, Jae-ho Lee classified up to isomorphism the circular Hessenberg systems that satisfy a certain condition. Lee conjectured that
the condition is superfluous; it is satisfied by any circular Hessenberg system. We now explain the condition.
\medskip

\noindent
For the rest of this paper,  fix  a  Hessenberg system
\begin{align*}
\Phi = (A; \lbrace E_i \rbrace_{i=0}^d; A^*; \lbrace E^*_i \rbrace_{i=0}^d)
\end{align*}
on $V$, with parameter array
\begin{align*}
\bigl( \lbrace \theta_i \rbrace_{i=0}^d; \lbrace \theta^*_i \rbrace_{i=0}^d; \lbrace \phi_i \rbrace_{i=1}^d\bigr).
\end{align*}

\begin{conjecture} \label{conj} \rm (See \cite[Conjecture~2.13]{JHL}.) Assume that the Hessenberg system $\Phi$ is circular. Then there exist scalars
 $\beta, \gamma, \gamma^*, \varrho, \varrho^*$  such that both
\begin{align}
0 &= \lbrack A, A^2 A^*-\beta A A^* A + A^* A^2 - \gamma(AA^*+A^* A) -\varrho A^* \rbrack,   \label{eq:CTD1} \\
0 &= \lbrack A^*, A^{*2} A-\beta A^* A A^* + A A^{*2} - \gamma^*(A^*A+A A^*) -\varrho^* A \rbrack. \label{eq:CTD2}
\end{align}
\end{conjecture}
\noindent The relations \eqref{eq:CTD1}, \eqref{eq:CTD2} are called the {\it tridiagonal relations}, see \cite[Lemma~5.4]{tSub3}, \cite{qSerre}.
\medskip

\noindent In \cite{JHL}, Jae-ho Lee classified up to isomorphism  the circular Hessenberg systems
that satisfy the condition in Conjecture \ref{conj}.
\medskip

\noindent Our goal in this paper is to prove Conjecture \ref{conj}. 
\medskip

\noindent  For $d=2$, the Conjecture \ref{conj} is implied by the following result about general Hessenberg systems.
\begin{lemma}  Assume $d=2$. Then $\Phi$ satisfies the tridiagonal relations  \eqref{eq:CTD1}, \eqref{eq:CTD2} with 
\begin{align*}
&\beta=-1, \qquad \quad \gamma=\theta_0 + \theta_1 + \theta_2, \qquad \qquad  \gamma^*=\theta^*_0 + \theta^*_1 + \theta^*_2, \\
&\varrho = - \theta_0 \theta_1 - \theta_0 \theta_2 - \theta_1 \theta_2, \qquad \qquad \varrho^* = - \theta^*_0 \theta^*_1 - \theta^*_0 \theta^*_2 - \theta^*_1 \theta^*_2.
\end{align*}
\end{lemma}
\begin{proof} Since $\Phi$ is a Hessenberg system, we may represent $A$ and $A^*$ as matrices
\begin{align*}
A = \begin{pmatrix}  \theta_2 &0 & 0 \\ 
                                    1 &\theta_1 &0\\
                                      0 &1&\theta_0 \\
                                       \end{pmatrix}, \qquad \qquad 
 A^* = \begin{pmatrix} \theta^*_0&\phi_1 &0\\ 
                                   0 &\theta^*_1&\phi_2 \\
                                    0&0&\theta^*_2
                                    \end{pmatrix}. 
\end{align*}
These matrices satisfy  \eqref{eq:CTD1}, \eqref{eq:CTD2} by matrix multiplication.
\end{proof}

\noindent We now give some results about $\Phi$ that will be used to prove Conjecture \ref{conj} for $d\geq 3$.
\medskip

\noindent For notational convenience, define
\begin{align}
\vartheta_i = \phi_i - (\theta^*_i - \theta^*_0)(\theta_{d-i+1}-\theta_0) \qquad \qquad (1 \leq i \leq d)        \label{eq:vth}
\end{align}
and
\begin{align*}
\vartheta_0 = 0, \qquad \qquad \vartheta_{d+1}=0.
\end{align*}
By Lemma \ref{lem:dual} and Definition \ref{def:dual},
\begin{align*}
\vartheta^*_i = \vartheta_{d-i+1} \qquad \qquad (0 \leq i \leq d+1).
\end{align*}

\noindent  We will be discussing sequences of scalars that are $\beta$-recurrent or  $(\beta, \gamma)$-recurrent or  $(\beta, \gamma, \varrho)$-recurrent.
See \cite[Section~8]{LS99} for the definitions and basic facts about these recurrences.

\begin{lemma} \label{lem:TD1} {\rm (See \cite[Lemma~12.5]{LS99}.)}
Assume $d\geq 3$ and pick  $\beta, \gamma, \varrho \in \mathbb F$.  Then 
\begin{align*}
0 = \lbrack A, A^2 A^*-\beta A A^* A + A^* A^2 - \gamma(AA^*+A^* A) -\varrho A^* \rbrack
\end{align*}
if and only if the following conditions hold:
\begin{enumerate}
\item[\rm (i)] the sequence $\lbrace \theta_i \rbrace_{i=0}^d$ is $(\beta, \gamma, \varrho)$-recurrent;
\item[\rm (ii)] the sequence $\lbrace \theta^*_i \rbrace_{i=0}^d$ is $\beta$-recurrent;
\item[\rm (iii)] the sequence $\lbrace \vartheta_i \rbrace_{i=0}^{d+1}$ is $\beta$-recurrent.
\end{enumerate}
\end{lemma}

\begin{lemma} \label{lem:TD2} 
 Assume $d\geq 3$ and pick $\beta, \gamma^*, \varrho^* \in \mathbb F$.  Then 
\begin{align*}
0 = \lbrack A^*, A^{*2} A-\beta A^* A A^* + A A^{*2} - \gamma^*(A^*A+A A^*) -\varrho^* A \rbrack
\end{align*}
if and only if the following conditions hold:
\begin{enumerate}
\item[\rm (i)] the sequence $\lbrace \theta^*_i \rbrace_{i=0}^d$ is $(\beta, \gamma^*, \varrho^*)$-recurrent;
\item[\rm (ii)] the sequence $\lbrace \theta_i \rbrace_{i=0}^d$ is $\beta$-recurrent;
\item[\rm (iii)] the sequence $\lbrace \vartheta_i \rbrace_{i=0}^{d+1}$ is $\beta$-recurrent.
\end{enumerate}
\end{lemma}
\begin{proof} Apply  Lemma \ref{lem:TD1} to $\Phi^*$.
\end{proof}

\begin{proposition}\label{prop:sum} Assume $d\geq 3$. Then for $\beta \in \mathbb F$ the following {\rm (i)--(iv)} are equivalent:
\begin{enumerate}
\item[\rm (i)] there exist $\gamma, \gamma^*, \varrho, \varrho^* \in \mathbb F$ such that both
\begin{align*}
0 &= \lbrack A, A^2 A^*-\beta A A^* A + A^* A^2 - \gamma(AA^*+A^* A) -\varrho A^* \rbrack,   \\
0 &= \lbrack A^*, A^{*2} A-\beta A^* A A^* + A A^{*2} - \gamma^*(A^*A+A A^*) -\varrho^* A \rbrack;
\end{align*}
\item[\rm (ii)] there exist $\gamma, \varrho \in \mathbb F$ such that
\begin{align*}
0 = \lbrack A, A^2 A^*-\beta A A^* A + A^* A^2 - \gamma(AA^*+A^* A) -\varrho A^* \rbrack;
\end{align*}
\item[\rm (iii)] there exist $ \gamma^*, \varrho^* \in \mathbb F$ such that
\begin{align*}
0 = \lbrack A^*, A^{*2} A-\beta A^* A A^* + A A^{*2} - \gamma^*(A^*A+A A^*) -\varrho^* A \rbrack;
\end{align*}
\item[\rm (iv)] each of
$\lbrace \theta_i \rbrace_{i=0}^d$, $\lbrace \theta^*_i \rbrace_{i=0}^d$, $\lbrace \vartheta_i \rbrace_{i=0}^{d+1}$ is $\beta$-recurrent.
\end{enumerate}
\end{proposition}
\begin{proof} Conditions (ii), (iv) are equivalent by Lemma \ref{lem:TD1} and  \cite[Section~8]{LS99}.
Conditions (iii), (iv) are equivalent by Lemma \ref{lem:TD2} and  \cite[Section~8]{LS99}. By these comments, conditions (ii)--(iv) are equivalent.
In particular, conditions (ii), (iii) are equivalent. By this and the construction, conditions (i)--(iii) are equivalent. The result follows.
\end{proof}
\noindent For $d\geq 3$, Conjecture \ref{conj} asserts that there exists $\beta \in \mathbb F$ such that the equivalent conditions (i)--(iv) in Proposition \ref{prop:sum}
all hold.

\section{The case $d=3$}

\noindent  We continue to discuss the Hessenberg system
\begin{align*}
\Phi = (A; \lbrace E_i \rbrace_{i=0}^d; A^*; \lbrace E^*_i \rbrace_{i=0}^d)
\end{align*}
on $V$, with parameter array 
\begin{align*}
\bigl( \lbrace \theta_i \rbrace_{i=0}^d; \lbrace \theta^*_i \rbrace_{i=0}^d; \lbrace \phi_i \rbrace_{i=1}^d\bigr).
\end{align*}
\noindent  In this section, we prove Conjecture \ref{conj} under the assumption $d=3$.

\begin{lemma} Assume that $\Phi$ is circular and $d=3$. Then the following {\rm (i)--(iii)}  hold.
\begin{enumerate}
\item[\rm (i)]  $\vartheta_1 \not=\vartheta_3$ and
\begin{align*}
&0 = \theta_0 - \theta_1 + \theta_2 - \theta_3, \qquad \qquad
0 = \theta^*_0 - \theta^*_1 + \theta^*_2 - \theta^*_3, \\
&0 = \vartheta_1 - \vartheta_2 + \vartheta_3. 
\end{align*}
\item[\rm (ii)]  Each of $\lbrace \theta_i \rbrace_{i=0}^d$, $\lbrace \theta^*_i \rbrace_{i=0}^d$, $\lbrace \vartheta_i \rbrace_{i=0}^{d+1}$ is
$0$-recurrent.
\item[\rm (iii)] The tridiagonal relations \eqref{eq:CTD1}, \eqref{eq:CTD2} hold with $\beta=0$ and
\begin{align*}
 \gamma &= \theta_0 + \theta_2 =\theta_1 + \theta_3,    \\
  \gamma^* &= \theta^*_0 + \theta^*_2 =\theta^*_1 + \theta^*_3,   \\
 \varrho &= (\theta_1 - \theta_0)(\theta_1-\theta_2) - 2 \theta_0 \theta_2 = (\theta_2 - \theta_1)(\theta_2 - \theta_3) - 2 \theta_1 \theta_3, \\
  \varrho^* &= (\theta^*_1 - \theta^*_0)(\theta^*_1-\theta^*_2) - 2 \theta^*_0 \theta^*_2 = (\theta^*_2 - \theta^*_1)(\theta^*_2 - \theta^*_3) - 2 \theta^*_1 \theta^*_3.
\end{align*}
\end{enumerate}
\end{lemma}
\begin{proof} Since $\Phi$ is a Hessenberg system, we may represent $A$ and $A^*$ as matrices
\begin{align}
A = \begin{pmatrix}  \theta_3 &0 & 0 &0 \\ 
                                    1 &\theta_2 &0& 0\\
                                      0 &1&\theta_1 &0 \\
                                       0&0&1&\theta_0
                                       \end{pmatrix}, \qquad \qquad 
 A^* = \begin{pmatrix} \theta^*_0&\phi_1 &0&0 \\ 
                                   0 &\theta^*_1&\phi_2&0 \\
                                    0&0&\theta^*_2& \phi_3\\
                                    0&0&0&\theta^*_3
                                    \end{pmatrix}. \label{eq:AAs}
\end{align}
By \eqref{eq;Eformula}, the primitive idempotents of $A$ are
\begin{align*}
E_0 &= \begin{pmatrix}
0&0&0&0 \\
0&0&0&0 \\
0&0&0&0 \\
\frac{1}{(\theta_0-\theta_1)(\theta_0-\theta_2)(\theta_0-\theta_3)}&\frac{1}{(\theta_0-\theta_1)(\theta_0-\theta_2)}&\frac{1}{\theta_0-\theta_1}&1
\end{pmatrix},
\\
E_1 &= \begin{pmatrix}
         0&0&0&0 \\
        0 &0&0&0 \\
      \frac{1}{(\theta_1-\theta_2)(\theta_1-\theta_3)}   &\frac{1}{\theta_1-\theta_2}&1&0 \\
 \frac{1}{(\theta_1-\theta_0)(\theta_1-\theta_2)(\theta_1-\theta_3)}        &\frac{1}{(\theta_1-\theta_0)(\theta_1-\theta_2)}&\frac{1}{\theta_1-\theta_0}&0
\end{pmatrix}, 
\\
E_2 &= \begin{pmatrix}
0&0&0&0 \\
\frac{1}{\theta_2-\theta_3}&1&0&0 \\
\frac{1}{(\theta_2-\theta_1)(\theta_2-\theta_3)}&\frac{1}{\theta_2 - \theta_1}&0&0 \\
\frac{1}{(\theta_2-\theta_0)(\theta_2-\theta_1)(\theta_2-\theta_3)}&\frac{1}{(\theta_2-\theta_0)(\theta_2-\theta_1)}&0&0
\end{pmatrix},
\\
E_3 &= \begin{pmatrix}
 1&0&0&0\\
\frac{1}{\theta_3-\theta_2} &0&0&0 \\
\frac{1}{(\theta_3-\theta_1)(\theta_3-\theta_2)} &0&0&0 \\
\frac{1}{(\theta_3-\theta_0)(\theta_3-\theta_1)(\theta_3-\theta_2)} &0&0&0
 \end{pmatrix}.
\end{align*}

\noindent (i) From the $(2,0)$-entry of $E_1 A^* E_3=0$ we obtain
\begin{align}
\begin{split}
0 &=(\theta_0 - \theta_1) \Bigl(-\theta^*_0 + \frac{\theta_0 - \theta_3}{ \theta_1 - \theta_2} \theta^*_1 -   \frac{\theta_0 - \theta_3}{ \theta_1 - \theta_2} \theta^*_2+ \theta^*_3   \Bigr)\\
& \qquad \qquad - \frac{\theta_0 - \theta_3}{ \theta_1 - \theta_2} \vartheta_1 + \frac{\theta_0 - \theta_3}{ \theta_1 - \theta_2} \vartheta_2 -\vartheta_3.
\end{split} \label{eq:n1}
\end{align}
From the $(3,1)$-entry of $E_0 A^* E_2=0$ we obtain
\begin{align}
0 = \vartheta_1 -  \frac{\theta_0 - \theta_3}{ \theta_1 - \theta_2} \vartheta_2 +  \frac{\theta_0 - \theta_3}{ \theta_1 - \theta_2} \vartheta_3. \label{eq:n2}
\end{align}
From the $(3,0)$-entry of $E_0 A^* E_3 \not=0$ we obtain
\begin{align}
0 \not=\vartheta_2 (\theta_0-\theta_3) -\vartheta_1(\theta_1-\theta_3)-\vartheta_3(\theta_0-\theta_2).   \label{eq:nn3}
\end{align}
Applying these results to $\Phi^*$ and using 
\begin{align*}
\vartheta^*_1 = \vartheta_3, \qquad \quad \vartheta^*_2 = \vartheta_2, \qquad \quad \vartheta^*_3 = \vartheta_1
\end{align*}
we obtain
\begin{align}
\begin{split}
0 &=(\theta^*_0 - \theta^*_1) \Bigl(-\theta_0 + \frac{\theta^*_0 - \theta^*_3}{ \theta^*_1 - \theta^*_2} \theta_1 -   \frac{\theta^*_0 - \theta^*_3}{ \theta^*_1 - \theta^*_2} \theta_2+ \theta_3   \Bigr) \\
&\qquad \qquad - \frac{\theta^*_0 - \theta^*_3}{ \theta^*_1 - \theta^*_2} \vartheta_3 + \frac{\theta^*_0 - \theta^*_3}{ \theta^*_1 - \theta^*_2} \vartheta_2 -\vartheta_1
\end{split}
\label{eq:n1s}
\end{align}
and 
\begin{align}
0 = \vartheta_3 -  \frac{\theta^*_0 - \theta^*_3}{ \theta^*_1 - \theta^*_2} \vartheta_2 +  \frac{\theta^*_0 - \theta^*_3}{ \theta^*_1 - \theta^*_2} \vartheta_1 \label{eq:n2s}
\end{align}
and
\begin{align}
0 \not=\vartheta_2 (\theta^*_0-\theta^*_3) -\vartheta_3(\theta^*_1-\theta^*_3)-\vartheta_1(\theta^*_0-\theta^*_2).       \label{eq:nn3s}
\end{align}
We now solve the equations and inequalities \eqref{eq:n1}--\eqref{eq:nn3s}. We first show $\vartheta_1 \not=\vartheta_3$.
Using \eqref{eq:n2}, \eqref{eq:nn3} we obtain
\begin{align*}
\vartheta_1 - \vartheta_3 &= \frac{\vartheta_1 (\theta_1 - \theta_2) + \vartheta_2 (\theta_3- \theta_0) + \vartheta_3 (\theta_0 - \theta_3)}{\theta_3 - \theta_2} \\
 & \qquad + \frac{\vartheta_1 (\theta_3 - \theta_1) + \vartheta_2 (\theta_0- \theta_3) + \vartheta_3 (\theta_2 - \theta_0)}{\theta_3 - \theta_2} 
\\
 & = \frac{\vartheta_1 (\theta_3 - \theta_1) + \vartheta_2 (\theta_0- \theta_3) + \vartheta_3 (\theta_2 - \theta_0)}{\theta_3 - \theta_2} \\
&\not= 0.
\end{align*}
For notational convenience, define
\begin{align}
\eta = \frac{\theta_0 - \theta_3}{\theta_1- \theta_2}, \qquad \qquad \eta^* =  \frac{\theta^*_0 - \theta^*_3}{\theta^*_1 - \theta^*_2},
 \qquad \qquad \vartheta = \vartheta_1 - \vartheta_2 + \vartheta_3 \label{eq:notate}
\end{align}
and note that $\eta, \eta^*$ are nonzero. We will show that
\begin{align*}
\eta = 1, \qquad \qquad \eta^*=1, \qquad \qquad \vartheta=0.
\end{align*}
Write \eqref{eq:n1}, \eqref{eq:n2} and \eqref{eq:n1s}, \eqref{eq:n2s} in terms of $\eta, \eta^*, \vartheta$.
 Equation \eqref{eq:n1} becomes
\begin{align}
0 = (\theta_0 - \theta_1)(\theta^*_1 - \theta^*_2) (\eta-\eta^*) - \eta \vartheta + (\eta-1) \vartheta_3.
\label{eq:nnt1}
\end{align}
 Equation \eqref{eq:n2} becomes
\begin{align}
0 = \eta \vartheta + (1-\eta) \vartheta_1.
\label{eq:nnt2}
\end{align}
 Equation \eqref{eq:n1s} becomes
\begin{align}
0 = (\theta^*_0 - \theta^*_1)(\theta_1 - \theta_2) (\eta^*-\eta) - \eta^* \vartheta + (\eta^*-1) \vartheta_1.
\label{eq:nnt1s}
\end{align}
 Equation \eqref{eq:n2s} becomes
\begin{align}
0 = \eta^* \vartheta + (1-\eta^*) \vartheta_3.
\label{eq:nnt2s}
\end{align}
If $\eta=1$, then $\vartheta=0$ by \eqref{eq:nnt2}  and $\eta^*=1$ by \eqref{eq:nnt1}.
If $\eta^*=1$, then  $\vartheta=0$ by \eqref{eq:nnt2s}  and $\eta=1$ by \eqref{eq:nnt1s}.
For the rest of this proof, assume that $\eta\not=1$ and $\eta^*\not=1$. We will get a contradiction. By \eqref{eq:nnt2}, \eqref{eq:nnt2s} we obtain
\begin{align}
\vartheta_1 = \frac{\eta}{\eta-1} \vartheta, \qquad \qquad \vartheta_3 = \frac{\eta^*}{\eta^*-1} \vartheta.
\label{eq:nnt3}
\end{align}
Evaluating \eqref{eq:nnt1}, \eqref{eq:nnt1s} using \eqref{eq:nnt3}, we get
\begin{align}
0 &= (\theta_0 - \theta_1)(\theta^*_1 - \theta^*_2) (\eta-\eta^*)  + \frac{\eta-\eta^*}{\eta^*-1} \vartheta,
\label{eq:nnnt1}
\\
0 &= (\theta^*_0 - \theta^*_1)(\theta_1 - \theta_2) (\eta^*-\eta)  +\frac{\eta^*-\eta}{\eta-1} \vartheta.
\label{eq:nnnt1s}
\end{align}
We have $\eta\not=\eta^*$; otherwise \eqref{eq:nnt3} gives $\vartheta_1 =\vartheta_3$  which is forbidden by our comments below  \eqref{eq:nn3s}.
Now \eqref{eq:nnnt1}, \eqref{eq:nnnt1s} become
\begin{align}
0 &= (\theta_0 - \theta_1)(\theta^*_1 - \theta^*_2)   + \frac{\vartheta}{\eta^*-1},
\label{eq:nnnnt1}
\\
0 &= (\theta^*_0 - \theta^*_1)(\theta_1 - \theta_2) +\frac{\vartheta}{\eta-1}.
\label{eq:nnnnt1s}
\end{align}
 Combining  \eqref{eq:vth}, \eqref{eq:notate}, \eqref{eq:nnt3}, \eqref{eq:nnnnt1}, \eqref{eq:nnnnt1s}
we get  the contradictions
\begin{align*}
0 \not= \frac{\phi_1}{\eta} &= \frac{ \vartheta_1 + (\theta^*_1-\theta^*_0)(\theta_3 - \theta_0)}{\eta} 
= \frac{\vartheta}{\eta-1} + (\theta^*_0 - \theta^*_1)(\theta_1 - \theta_2) = 0
\end{align*}
and 
\begin{align*}
0 \not= \frac{\phi_3}{\eta^*} &= \frac{ \vartheta_3 + (\theta^*_3-\theta^*_0)(\theta_1 - \theta_0)}{\eta^*} 
= \frac{\vartheta}{\eta^*-1} + (\theta_0 - \theta_1)(\theta^*_1 - \theta^*_2) = 0.
\end{align*}
\\
\noindent (ii) By (i) above. \\
\noindent (iii) By matrix multiplication using \eqref{eq:AAs} and (i) above.
\end{proof}

\section{The case $d\geq 4$}
\noindent  We continue to discuss the Hessenberg system
\begin{align*}
\Phi = (A; \lbrace E_i \rbrace_{i=0}^d; A^*; \lbrace E^*_i \rbrace_{i=0}^d)
\end{align*}
on $V$, with parameter array 
\begin{align*}
\bigl( \lbrace \theta_i \rbrace_{i=0}^d; \lbrace \theta^*_i \rbrace_{i=0}^d; \lbrace \phi_i \rbrace_{i=1}^d\bigr).
\end{align*}
\noindent  In this section, we prove Conjecture \ref{conj} under the assumption $d\geq 4$.
\medskip

\noindent We will use  the following strategy.
Throughout this section, we assume that (i) $d\geq 4$; (ii) $\Phi$ is circular; (iii) for all $\beta \in \mathbb F$ the equivalent conditions (i)--(iv) in Proposition \ref{prop:sum}
do not hold for $\Phi$. We will get a contradiction.
\medskip

\noindent As we go along, keep in mind that for every result we obtain about $\Phi$, there is a  similar result about $\Phi^*$.
\medskip

\noindent Let $\langle A \rangle$ denote the subalgebra of ${\rm End}(V)$ generated by $A$. Note that  $\langle A \rangle$  has a basis $\lbrace A^i \rbrace_{i=0}^d$
and a basis $\lbrace E_i \rbrace_{i=0}^d$.

\begin{lemma} \label{lem:lem1}
For $0 \leq i \leq d-1$,
\begin{align*}
(E_0 +& E_1 + \cdots + E_i ) A^* - A^* (E_0 + E_1+ \cdots + E_i) \\
&= E_i A^* E_{i+1}-E_{i+1} A^* E_i + E_0 A^* E_d.
\end{align*}
\end{lemma}
\begin{proof} To verify this equation, evaluate the left-hand side using
\begin{align*}
 E_0 A^* &= E_0 A^* E_0 + E_0 A^* E_1 + E_0 A^* E_d, \\
E_j A^* &= E_j A^* E_{j-1} + E_j A^* E_j + E_j A^* E_{j+1} \qquad (1 \leq j \leq i), \\
 A^* E_0 &= E_0 A^* E_0 + E_1 A^* E_0, \\
 A^* E_j &= E_{j-1} A^* E_j + E_j A^* E_j + E_{j+1} A^* E_j \qquad (1 \leq j \leq i). 
 \end{align*}
\end{proof}

\begin{lemma}\label{lem:lem2} For $r,s \in \mathbb N$ we have
\begin{align*}
A^r A^* A^s -A^s A^* A^r &= \sum_{i=0}^{d-1} (\theta_i^r \theta_{i+1}^s - \theta_i^s \theta_{i+1}^r) \Bigl(  (E_0  + \cdots + E_i ) A^* - A^* (E_0 + \cdots + E_i)  \Bigr)\\
&\qquad + \Psi(r,s) E_0 A^* E_d,
\end{align*}
where
\begin{align*}
\Psi(r,s) &= \theta_0^r \theta_d^s - \theta_0^s \theta_d^r- \sum_{i=0}^{d-1} (\theta_i^r \theta_{i+1}^s - \theta_i^s \theta_{i+1}^r).
\end{align*}
\end{lemma}
\begin{proof}
We have
\begin{align*}
A^r &A^* A^s -A^s A^* A^r \\
&= \sum_{i=0}^d \sum_{j=0}^d E_i (A^r A^* A^s -A^s A^* A^r ) E_j \\
               &= \sum_{i=0}^d \sum_{j=0}^d (\theta_i^r \theta_j^s-\theta_i^s \theta_j^r)E_i A^* E_j \\
               &=    (\theta_0^r \theta_d^s-\theta_0^s \theta_d^r)E_0A^* E_d +
                     \sum_{i=0}^{d-1} (\theta_i^r \theta_{i+1}^s-\theta_i^s \theta_{i+1}^r)E_i A^*E_{i+1}
                 + \sum_{i=0}^{d-1} (\theta_{i+1}^r \theta_{i}^s-\theta_{i+1}^s \theta_i^r)E_{i+1} A^*E_i \\
               &=    (\theta_0^r \theta_d^s-\theta_0^s \theta_d^r)E_0A^* E_d +
                     \sum_{i=0}^{d-1} (\theta_i^r \theta_{i+1}^s-\theta_i^s \theta_{i+1}^r)(E_i A^*E_{i+1} - E_{i+1} A^* E_i).
\end{align*}
In this sum, for $0 \leq i \leq d-1$ we eliminate $E_i A^*E_{i+1} - E_{i+1} A^* E_i$ using Lemma \ref{lem:lem1}.
The result follows.
\end{proof}

\begin{lemma} \label{lem:lem3} We have $\Psi(2,1) \not=0$.
\end{lemma}
\begin{proof} We assume that $\Psi(2,1)=0$ and get a contradiction.
By Lemma \ref{lem:lem2}, there exists $X \in \langle A \rangle$  such that
\begin{align*}
A^2 A^* A - A A^* A^2  = X A^*-A^*X.
\end{align*}
Write $X=f(A)$ where $f \in \mathbb F\lbrack \lambda \rbrack$ has degree at most $d$. So
\begin{align}
A^2 A^* A - A A^* A^2  = f(A) A^*-A^*f(A). \label{eq:eqf}
\end{align}
Let $s$ denote the degree of $f$. We show that $s=3$. First assume that $s>3$. Let $c$ denote the leading coefficient in $f$, and note that $c \not=0$.
We have
\begin{align*}
0 &= E^*_s \Bigl( A^2 A^* A - A A^* A^2  - f(A) A^*+A^*f(A) \Bigr) E^*_0 \\
&= E^*_s A^s E^*_0 c (\theta^*_s - \theta^*_0) \\
&= E^*_s A E^*_{s-1}  A \cdots A E^*_1 A E^*_0 c (\theta^*_s- \theta^*_0)\\
& \not=0,
\end{align*}
for a contradiction. Next assume that $s<3$. We have
\begin{align*}
0 &= E^*_3 \Bigl( A^2 A^* A - A A^* A^2  - f(A) A^*+A^*f(A) \Bigr) E^*_0 \\
&= E^*_3 A^3 E^*_0  (\theta^*_1 - \theta^*_2) \\
&= E^*_3  A E^*_2 A E^*_1 A E^*_0  (\theta^*_1 - \theta^*_2)\\
& \not=0,
\end{align*}
for a contradiction.
We have shown that $s=3$. For notational convenience, define $\beta = c^{-1}-1$. In the equation \eqref{eq:eqf}, divide
each side by $c$ and evaluate the result using $c^{-1} = \beta+1$. There exist $\gamma, \varrho \in \mathbb F$
such that
\begin{align*}
A^3 A^* - (\beta+1) A^2 A^* A + (\beta+1) A A^* A^2 - A^* A^3 = \gamma(A^2 A^* - A^* A^2) + \varrho (A A^*- A^* A).
\end{align*}
Rearranging the terms in this equation, we obtain
\begin{align*}
0 = \lbrack A, A^2 A^*-\beta A A^* A + A^* A^2 - \gamma(AA^*+A^* A) -\varrho A^* \rbrack.
\end{align*}
Consequently Proposition \ref{prop:sum}(ii) holds for $\beta$. This contradicts an assumption we made in our strategy, so  $\Psi(2,1)\not=0$.
\end{proof}

\begin{lemma} \label{lem:lem4} The element $A^4 A^*-A^* A^4$ is contained in the span of
\begin{align*}
&AA^*-A^*A, \qquad \quad 
A^2 A^*-A^*A^2, \qquad \quad
A^3 A^*-A^*A^3, \\
&A^2 A^*A-A A^* A^2, \qquad \quad
A^3 A^*A-A A^* A^3.
\end{align*}
\end{lemma}
\begin{proof} 
By Lemmas \ref{lem:lem2},  \ref{lem:lem3} we have
\begin{align*}
&A^3 A^* A -A A^* A^3 - \frac{\Psi(3,1)}{\Psi(2,1)} \Bigl( A^2 A^* A - A A^* A^2\Bigr) \\
&= \sum_{i=0}^{d-1} (\theta_i^3 \theta_{i+1} - \theta_i \theta_{i+1}^3) \Bigl(  (E_0  + \cdots + E_i ) A^* - A^* (E_0 + \cdots + E_i)  \Bigr) \\
 &\qquad -     \frac{\Psi(3,1)}{\Psi(2,1)}  \sum_{i=0}^{d-1} (\theta_i^2 \theta_{i+1} - \theta_i \theta_{i+1}^2) \Bigl(  (E_0  + \cdots + E_i ) A^* - A^* (E_0 + \cdots + E_i)         \Bigr).
\end{align*}
Consequently, there exists $Y \in  \langle A \rangle$  such that
\begin{align*}
&A^3 A^* A -A A^* A^3 - \frac{\Psi(3,1)}{\Psi(2,1)} \Bigl( A^2 A^* A - A A^* A^2\Bigr) =YA^*-A^*Y.
\end{align*}
Write $Y=g(A)$ where $g \in \mathbb F\lbrack \lambda \rbrack$ has degree at most $d$. So
\begin{align}
&A^3 A^* A -A A^* A^3 - \frac{\Psi(3,1)}{\Psi(2,1)} \Bigl( A^2 A^* A - A A^* A^2\Bigr) =g(A)A^*-A^*g(A).  \label{eq:geqf}
\end{align}
Let $t$ denote the degree of $g$. We show that $t=4$. First assume that $t>4$. Let $k$ denote the leading coefficient in $g$, and note that $k \not=0$.
We have
\begin{align*}
0 &= E^*_t \biggl(  A^3 A^* A -A A^* A^3 - \frac{\Psi(3,1)}{\Psi(2,1)} \Bigl( A^2 A^* A - A A^* A^2\Bigr) -g(A)A^*+A^*g(A)                   \biggr) E^*_0 \\
&= E^*_t A^t E^*_0 k (\theta^*_t - \theta^*_0) \\
&= E^*_t  A E^*_{t-1}  A \cdots A E^*_1 A E^*_0 k (\theta^*_t- \theta^*_0)\\
& \not=0,
\end{align*}
for a contradiction. Next assume that $t<4$. We have
\begin{align*}
0 &= E^*_4 \biggl(  A^3 A^* A -A A^* A^3 - \frac{\Psi(3,1)}{\Psi(2,1)} \Bigl( A^2 A^* A - A A^* A^2\Bigr) -g(A)A^*+A^*g(A)             \biggr) E^*_0 \\
&= E^*_4 A^4 E^*_0  (\theta^*_1 - \theta^*_3) \\
&= E^*_4 A E^*_3 A E^*_2 A E^*_1 A E^*_0  (\theta^*_1 - \theta^*_3)\\
& \not=0,
\end{align*}
for a contradiction.
We have shown that $t=4$.
The result follows from \eqref{eq:geqf} and the fact that $g$ has degree $4$.
\end{proof}

\begin{lemma} \label{lem:lem5} There exists $\beta \in \mathbb F$ such that
\begin{align*}
A^4 A^*-A^* A^4 -\beta (A^3 A^* A- A A^* A^3) 
\end{align*}
 is contained in the span of
\begin{align}
&AA^*-A^*A, \qquad 
A^2 A^*-A^*A^2, \qquad 
A^3 A^*-A^*A^3, \qquad 
A^2 A^*A-A A^* A^2.
\label{eq:span5}
\end{align}
\end{lemma}
\begin{proof} By Lemma \ref{lem:lem4}.
\end{proof}

\noindent We bring in some notation.
For $i \in \mathbb Z$  define 
\begin{align*}
\theta^*_i = \theta^*_r, \qquad \qquad E^*_i = E^*_r
\end{align*}
where $r$ is the remainder upon dividing $i$ by $d+1$.
Note that $0 \leq r \leq d$, and $d+1$ divides $i-r$. 
By construction, for $i \in \mathbb Z$ we have
\begin{align}
\theta^*_i = \theta^*_{i+d+1}, \qquad \qquad E^*_i = E^*_{i+d+1}. \label{eq:mod}
\end{align}
\noindent Our next general goal is to show that
\begin{align*}
\theta^*_{i-3} - \theta^*_{i+1} = \beta (\theta^*_{i-2}-\theta^*_i) \qquad \qquad i \in \mathbb Z,
\end{align*}
\noindent  where $\beta$ is from Lemma \ref{lem:lem5}. To reach our goal, it is convenient to adjust our point of view. We will describe $\Phi$ using a directed graph.
\begin{definition} \label{def1} \rm
 Define a set $\mathcal D= \lbrace 0,1,\ldots, d\rbrace$. An element in $\mathcal D$ is called a {\it vertex}.
We turn $\mathcal D$ into a directed graph as follows. For vertices $i,j\in \mathcal D$ we draw a directed arc from $i$ to $j$ (written $i \rightarrow j$) whenever $E^*_j A E^*_i \not=0$.
\end{definition}
\begin{remark} \label{rem2} \rm
The graph $\mathcal D$ is described as follows.
We have   $d \rightarrow 0$.   For $1 \leq i \leq d$ we have     $i-1 \rightarrow i$ 
and we might have $i \rightarrow i-1$. For $0 \leq i \leq d$ we might have $i \rightarrow i$.
The graph $\mathcal D$ has no further arcs.
\end{remark}
\begin{definition}\label{def3} \rm
For $n \in \mathbb N$, a {\it walk of length $n$} in $\mathcal D$ is a sequence of vertices $\lbrace x_i \rbrace_{i=0}^n$ 
such that $x_0 \rightarrow x_1 \rightarrow \cdots \rightarrow x_n$. This walk is said to be {\it from $x_0$ to $x_n$}.
\end{definition}
\begin{lemma}\label{lem:nz}
Pick $n\in \mathbb N$ and  $x_0, x_1, \ldots, x_n \in \mathcal D$. The sequence $\lbrace x_i \rbrace_{i=0}^n$ 
is a walk in $\mathcal D$ if and only if
\begin{align*}
E^*_{x_n} A E^*_{x_{n-1}} A \cdots A E^*_{x_1} A E^*_{x_0} \not=0.
\end{align*}
\end{lemma}
\begin{proof} By Definition \ref{def1} and since $E^*_\ell$ has rank $1$ for $0 \leq \ell \leq d$.
\end{proof}

\begin{definition}\label{def6} \rm
For a walk $\lbrace x_i \rbrace_{i=0}^n$ 
 in $\mathcal D$, its
 {\it winding number} is
\begin{align*}
\bigl \vert \lbrace i \vert 1 \leq i \leq n, \; \; x_{i-1} = d, \; x_i =0 \rbrace \bigr\vert.
\end{align*}
\end{definition}
\begin{remark} \label{lem:shortwalk}
A walk in $\mathcal D$ of length at most $d$ has winding number $0$ or $1$.
\end{remark}
\begin{definition} \label{def16} \rm  A walk $\lbrace x_i \rbrace_{i=0}^n$ in $\mathcal D$ is called {\it invertible}
whenever the sequence $\lbrace x_{n-i} \rbrace_{i=0}^n$ is a walk in $\mathcal D$. In this case, we say that the walks 
$\lbrace x_i \rbrace_{i=0}^n$
 and
$\lbrace x_{n-i} \rbrace_{i=0}^n$
 are {\it inverses}.
\end{definition} 
\begin{remark} An invertible walk has winding number 0.
\end{remark}

\begin{definition}\label{def4} \rm
A {\it loop} in $\mathcal D$ is a walk of the form $i,i$.
A {\it backtrack} in $\mathcal D$ is a walk of the form $i,j,i$ with $i \not=j$.
\end{definition}
\begin{remark}\label{rem5} \rm
For distinct $i,j \in \mathcal D$ the sequence $i,j,i$ is a backtrack  if and only if the sequence $j,i,j$ is a backtrack.
\end{remark}
\begin{definition}\label{def7}\rm For $n \in \mathbb N$,
 a {\it path of length $n$} in $\mathcal D$ is a walk $\lbrace x_i \rbrace_{i=0}^n$
  in $\mathcal D$ such that $x_{i-1} \not=x_i$ for $1 \leq i \leq n$ and
$x_{i-1} \not=x_{i+1} $ for $1 \leq i \leq n-1$.  Thus, a path  is a walk that has no loops and no backtracks.
\end{definition}

\begin{remark}\label{rem8} \rm
Every path in $\mathcal D$ has the form
\begin{align*}
& i \rightarrow i+1 \rightarrow \cdots \rightarrow d \rightarrow 0 \rightarrow 1 \rightarrow \cdots \rightarrow d \rightarrow 0 \rightarrow 1 \rightarrow \cdots \\
&\qquad \cdots \rightarrow d \rightarrow 0 \rightarrow 1\rightarrow \cdots \rightarrow d \rightarrow 0 \rightarrow 1\rightarrow  \cdots \rightarrow j
\end{align*}
with $0 \leq i,j\leq d$ or
\begin{align*}
j \rightarrow j-1 \rightarrow \cdots \rightarrow i+1 \rightarrow i
\end{align*}
with $0 \leq i \leq j \leq d$. 
\end{remark}

\begin{definition}\label{def11}\rm
Given a walk $\lbrace x_i \rbrace_{i=0}^n$ in $\mathcal D$, we cancel all loops and backtracks to obtain a path $\lbrace y_i \rbrace_{i=0}^m$   in $\mathcal D$
such that $m \leq n$ and $y_0=x_0$ and $y_m=x_n$. The path  $\lbrace y_i \rbrace_{i=0}^m$ is said to {\it support} the walk $\lbrace x_i \rbrace_{i=0}^n$.
\end{definition}
\begin{remark}\label{rem12}\rm
A walk and its supporting path have the same winding number.
\end{remark}
\begin{lemma} \label{wandp} Let $\lbrace x_i \rbrace_{i=0}^n$   denote a walk in $\mathcal D$, and let $\lbrace y_i \rbrace_{i=0}^m$ denote its supporting path.
Then $\lbrace x_i \rbrace_{i=0}^n$ is invertible if and only if  $\lbrace y_i \rbrace_{i=0}^m$ is invertible. In this case, 
the walk $\lbrace x_{n-i} \rbrace_{i=0}^n$ is supported by the path $\lbrace y_{n-i} \rbrace_{i=0}^m$.
\end{lemma}
\begin{proof} By Definition \ref{def11} and the construction.
\end{proof}

\noindent  The following definition is for later use.
\begin{definition}\label{def:pp} \rm 
For a walk (resp. path) $\lbrace x_i \rbrace_{i=0}^n$ in $\mathcal D$, the associated {\it walk-product} (resp. {\it path-product}) is the element
\begin{align*}
E^*_{x_n} A E^*_{x_{n-1}} A \cdots A E^*_{x_1} A E^*_{x_0}.
\end{align*}
\end{definition}

\noindent Shortly we will define the weight of a walk. We will use the following result.

\begin{lemma}\label{lem14}
 For $0 \leq i \leq d$ we have
\begin{align*}
 E^*_i A E^*_i = a_i E^*_i
 \end{align*}
 where
\begin{align*}
a_i = {\rm tr}(E^*_i A) = {\rm tr}(A E^*_i). 
\end{align*}
For $1 \leq i \leq d$ we have
\begin{align*}
E^*_i A E^*_{i-1} A E^*_i = x_i E^*_i, \qquad \qquad E^*_{i-1} A E^*_i A E^*_{i-1} = x_i E^*_{i-1}
\end{align*}
where
\begin{align*}
x_i = {\rm tr}(E^*_i A E^*_{i-1} A) = {\rm tr}(A E^*_{i-1} A E^*_i) = {\rm tr}(E^*_{i-1} A E^*_i A)  ={\rm tr}(A E^*_i A E^*_{i-1}). 
\end{align*}
\end{lemma}
\begin{proof} This is a variation on \cite[Definition~7.1]{qrac} and \cite[Lemma~7.5]{qrac}.
\end{proof}

\begin{definition}\label{def15} \rm Let $\lbrace x_i \rbrace_{i=0}^n$ 
 denote a walk  in $\mathcal D$. By Lemma \ref{lem:nz} the associated walk-product is nonzero:
\begin{align}
E^*_{x_n} A E^*_{x_{n-1}} A \cdots A E^*_{x_1} A E^*_{x_0}. \label{eq:cp}
\end{align}
Reducing the product \eqref{eq:cp} using Lemma \ref{lem14},
we find that the product \eqref{eq:cp} is equal to 
a nonzero scalar multiple of the path-product
\begin{align*}
E^*_{y_m} A E^*_{y_{m-1}} A \cdots A E^*_{y_1} A E^*_{y_0}
\end{align*}
associated with the path $\lbrace y_i \rbrace_{i=0}^m$ in $\mathcal D$ that supports the walk $\lbrace x_i \rbrace_{i=0}^n$. The nonzero scalar multiple is called the {\it weight} of the walk $\lbrace x_i \rbrace_{i=0}^n$.
\end{definition}
\begin{remark}\label{rem16} \rm
Every path in $\mathcal D$ has weight 1.
\end{remark}

\begin{lemma} \label{lem17} Inverse walks have the same weight.
\end{lemma}
\begin{proof}  By Lemma \ref{wandp} and Definition \ref{def15}.
\end{proof}

\begin{lemma} \label{prop:rsij} Consider the products
\begin{align}
E^*_i A^r A^* A^s E^*_j, \qquad \qquad E^*_j A^s A^* A^r E^*_i        \label{eq:2prod}
\end{align}
with $0 \leq i < j \leq d$ and $r,s \in \mathbb N$ and $r+s\leq d$.
Then there exist $\alpha_1, \alpha_2  \in \mathbb F$ such that
\begin{align*}
E^*_i A^r A^* A^s E^*_j &= \alpha_1 E^*_i A E^*_{i+1} A \cdots A E^*_{j-1} A E^*_j   \\
      & \qquad  + \alpha_2 E^*_i A E^*_{i-1} A \cdots A E^*_1 A E^*_0 A E^*_d A E^*_{d-1} A \cdots A E^*_{j+1} A E^*_j, \\
      E^*_j A^s A^* A^r E^*_i &= \alpha_1 E^*_j A E^*_{j-1} A \cdots A E^*_{i+1} A E^*_i.
\end{align*}
\end{lemma} 
\begin{proof} First, consider the product on the left in \eqref{eq:2prod}. Using $A^*= \sum_{\ell=0}^d \theta^*_\ell E^*_\ell$ we obtain
\begin{align*}
E^*_i A^r A^* A^s E^*_j = \sum_{\ell=0}^d \theta^*_\ell E^*_i A^r E^*_\ell A^s E^*_j.
\end{align*}
For $0 \leq \ell \leq d$ we write
\begin{align*}
E^*_i A^r E^*_\ell A^s E^*_j = E^*_i A I A I \cdots I A E^*_\ell A I A I \cdots I A E^*_j.
\end{align*}
In this product, eliminate each copy of $I$ using $I = E^*_0 + E^*_1 + \cdots + E^*_d$.
By these comments and Lemma \ref{lem:nz},  $E^*_i A^r A^* A^s E^*_j$ becomes a linear combination of walk-products. This linear combination is over the walks of length $r+s$  from vertex $j$ to vertex $i$.
These walks have winding number $0$ or $1$, by Remark \ref{lem:shortwalk}. For each walk of length $r+s$ from vertex $j$ to vertex $i$, the associated
 walk-product  is equal to the weight of the walk times the path-product associated with the path that supports the
walk. 
There exists at most one path from vertex $j$ to vertex $i$ that has winding number $0$,
and a unique path from vertex $j$ to vertex $i$ that has winding number 1. They are
\begin{align*}
 j \rightarrow j-1 \rightarrow \cdots \rightarrow i, \qquad \qquad 
j \rightarrow j+1 \rightarrow \cdots \rightarrow d \rightarrow 0 \rightarrow 1 \rightarrow \cdots \rightarrow  i.
\end{align*}
 Assume for the moment that the path  $j \rightarrow j-1 \rightarrow \cdots \rightarrow i$ exists. Then
 there exist $\alpha'_1, \alpha_2  \in \mathbb F$ such that
\begin{align*}
E^*_i A^r A^* A^s E^*_j &= \alpha'_1 E^*_i A E^*_{i+1} A \cdots A E^*_{j-1} A E^*_j   \\
      & \qquad  + \alpha_2 E^*_i A E^*_{i-1} A \cdots   A E^*_1 A E^*_0 A E^*_d A E^*_{d-1} A \cdots A E^*_{j+1} A E^*_j.
\end{align*}
 Next assume  that the path  $j \rightarrow j-1 \rightarrow \cdots \rightarrow i$ does not exist. Then
  \begin{align}
 E^*_i A E^*_{i+1} A \cdots A E^*_{j-1} A E^*_j =0 \label{eq:whatif}
 \end{align}
and there exists $ \alpha_2  \in \mathbb F$ such that
\begin{align}
E^*_i A^r A^* A^s E^*_j &=  \alpha_2 E^*_i A E^*_{i-1} A \cdots   A E^*_1 A E^*_0 A E^*_d A E^*_{d-1} A \cdots A E^*_{j+1} A E^*_j. \label{eq:whatif2}
\end{align}
We now consider the product on the right in  \eqref{eq:2prod}. As above, we express 
 $E^*_j A^s A^* A^r E^*_i$ as a linear combination of walk-products. This linear combination is over the walks of length $r+s$  from vertex $i$ to vertex $j$.
These walks have winding number $0$. For each walk of length $r+s$ from vertex $i$ to vertex $j$, the associated
 walk-product is equal to the weight of the walk times the path-product associated with the path that supports the
walk. 
There exists a unique path from vertex $i$ to vertex $j$ that has winding number $0$; it is
\begin{align*}
 i \rightarrow i+1 \rightarrow \cdots \rightarrow j.
\end{align*}
By these comments, there exists $\alpha_1 \in \mathbb F$ such that
\begin{align*}
 E^*_j A^s A^* A^r E^*_i &= \alpha_1 E^*_j A E^*_{j-1} A \cdots A E^*_{i+1} A E^*_i.
 \end{align*}
 We now discuss $\alpha_1, \alpha'_1$.
 For the moment, assume that the path  $j \rightarrow j-1 \rightarrow \cdots \rightarrow i$ exists. Then
 $\alpha_1=\alpha'_1$ in view of Lemmas \ref{wandp}, \ref{lem17}.
 Next assume that the path $j \rightarrow j-1 \rightarrow \cdots \rightarrow i$ does not exist. In this case \eqref{eq:whatif}, \eqref{eq:whatif2} hold,
 so
 \begin{align*}
E^*_i A^r A^* A^s E^*_j &= \alpha_1 E^*_i A E^*_{i+1} A \cdots A E^*_{j-1} A E^*_j   \\
      & \qquad  + \alpha_2 E^*_i A E^*_{i-1} A \cdots   A E^*_1 A E^*_0 A E^*_d A E^*_{d-1} A \cdots A E^*_{j+1} A E^*_j.
\end{align*}
The result follows.
\end{proof}

\begin{lemma}
\label{lem:lem6a} For $i \in \mathbb Z$,
\begin{align}
\theta^*_{i-3} - \theta^*_{i+1} = \beta (\theta^*_{i-2}-\theta^*_i). \label{eq:desire1}
\end{align}
\end{lemma}
\begin{proof} By Lemma \ref{lem:lem5}, there exists $ \gamma_\ell \in  \mathbb F$ $(1 \leq \ell \leq 4)$ such that
\begin{align*}
&A^4 A^*-A^* A^4 -\beta (A^3 A^* A- A A^* A^3) \\
&\qquad = \gamma_1 (A A^*-A^*A) + \gamma_2 (A^2 A^*-A^*A^2) + \gamma_3 (A^3 A^*-A^*A^3) + \gamma_4 (A^2 A^* A-A A^* A^2).
\end{align*}
For this equation, let $Z$ denote the left-hand side minus the right-hand side. So $Z=0$. We have
\begin{align*}
0 = E^*_{i-3} Z E^*_{i+1}, \qquad \qquad \quad 0 = E^*_{i+1} Z E^*_{i-3}.
\end{align*}
By the discussion above \eqref{eq:mod}, 
without loss of generality we may assume that $0 \leq i \leq d$. 
We now break the argument into two cases.
\medskip

\noindent {Case I}:  $3 \leq i \leq d-1$. Note that $i-3, i+1 \in \mathcal D$ and $i-3<i+1$.
By Lemma \ref{prop:rsij}, there exists $a_1 \in \mathbb F$ such that
\begin{align*}
E^*_{i+1} Z E^*_{i-3} =  a_1 E^*_{i+1} A E^*_i A E^*_{i-1} A E^*_{i-2} A E^*_{i-3}.
\end{align*}
By the form of $Z$,
\begin{align*}
a_1 = \theta^*_{i-3} - \theta^*_{i+1} -\beta (\theta^*_{i-2}-\theta^*_i). 
\end{align*}
By Lemma \ref{lem:nz},
\begin{align*}
0 \not=E^*_{i+1} A E^*_i A E^*_{i-1} A E^*_{i-2} A E^*_{i-3}.
\end{align*}
By these comments, $a_1=0$. Therefore \eqref{eq:desire1} holds for Case I.
\medskip

\noindent Case II: $i=d$ or $0 \leq i \leq 2$. By Lemma \ref{prop:rsij}, there exist $a_1, a_2 \in \mathbb F$
such that
\begin{align*}
E^*_{i+1} Z E^*_{i-3} &= a_1 E^*_{i+1} A E^*_{i+2} A \cdots A E^*_{i-4} A E^*_{i-3}   \\
      & \qquad  + a_2 E^*_{i+1} A E^*_{i} A \cdots A E^*_1 A E^*_0 A E^*_d A E^*_{d-1} A \cdots A E^*_{i-2} A E^*_{i-3}, \\
     E^*_{i-3} Z E^*_{i+1} &=  -a_1 E^*_{i-3} A E^*_{i-4} A \cdots A E^*_{i+2} A E^*_{i+1}.
\end{align*}
By the form of $Z$,
\begin{align*}
a_2 = \theta^*_{i-3} - \theta^*_{i+1} -\beta (\theta^*_{i-2}-\theta^*_i). 
\end{align*}
By Lemma \ref{lem:nz}, both
\begin{align*}
0 &\not= E^*_{i+1} A E^*_{i} A \cdots A E^*_1 A E^*_0 A E^*_d A E^*_{d-1} A \cdots A E^*_{i-2} A E^*_{i-3},\\
 0 &\not= E^*_{i-3} A E^*_{i-4} A \cdots A E^*_{i+2} A E^*_{i+1}.
\end{align*}
by these comments, $a_1=0$ and $a_2=0$. Therefore \eqref{eq:desire1} holds for Case II.
\end{proof}


\begin{lemma} \label{lem:6extra}
The scalar $\beta$ from Lemma \ref{lem:lem5} is unique.
\end{lemma}
\begin{proof} By Lemma \ref{lem:lem6a} and since $\lbrace \theta^*_i\rbrace_{i=0}^{d}$ are mutually distinct.
\end{proof}

\begin{lemma} \label{lem:betnz}
We have $\beta \not=0$.
\end{lemma}
\begin{proof} Suppose that $\beta=0$. Setting $i=3$ in Lemma \ref{lem:lem6a}, we obtain $\theta^*_0 = \theta^*_4$.
This contradicts the fact that  $\lbrace \theta^*_i \rbrace_{i=0}^d$ are mutually distinct. Therefore, $\beta\not=0$.
\end{proof}

\begin{definition} \label{def:check1} \rm
For $i \in \mathbb Z$ define
\begin{align*}
(\theta^*_i)^\vee = \theta^*_{i-2} - \theta^*_{i+1} - (\beta+1)(\theta^*_{i-1}-\theta^*_i).
\end{align*}
\end{definition}

\begin{lemma}
\label{lem:checksum} We have
\begin{align*}
0 = \sum_{i=0}^d (\theta^*_i)^\vee.
\end{align*}
\end{lemma}
\begin{proof} By \eqref{eq:mod} we have
\begin{align*}
\sum_{i=0}^d \theta^*_{i-2} = \sum_{i=0}^d \theta^*_{i+1} =\sum_{i=0}^d \theta^*_{i-1} =\sum_{i=0}^d \theta^*_{i}.
\end{align*}
The result follows in view of
Definition \ref{def:check1}.
\end{proof}

\begin{lemma} \label{lem:check2}
For $i \in \mathbb Z$,
\begin{align*}
(\theta^*_{i-1})^\vee + (\theta^*_i)^\vee =0.
\end{align*}
\end{lemma}
\begin{proof} To verify this equation, eliminate $(\theta^*_{i-1})^\vee$ and $(\theta^*_i)^\vee$ using Definition \ref{def:check1}, and evaluate the result using
Lemma \ref{lem:lem6a}.
\end{proof}

\begin{lemma} \label{lem:check3}
There exists $\xi \in \mathbb F$ such that 
\begin{align*}
(\theta^*_i)^\vee = (-1)^i \xi \qquad \qquad (i \in \mathbb Z).
\end{align*}
\end{lemma}
\begin{proof} By Lemma \ref{lem:check2}.
\end{proof}

\begin{lemma} \label{lem:check4} We have $\xi \not=0$.
\end{lemma}
\begin{proof} We assume $\xi =0$ and get a contradiction. By Definition \ref{def:check1},
\begin{align*}
 \theta^*_{i-2} - \theta^*_{i+1} = (\beta+1)(\theta^*_{i-1}-\theta^*_i) \qquad \qquad (i \in \mathbb Z).
 \end{align*}
 In this equation we rearrange terms to find that for  $i \in \mathbb Z$,
 \begin{align*}
 \theta^*_{i-2}-\beta \theta^*_{i-1} + \theta^*_i = \theta^*_{i-1} - \beta \theta^*_i + \theta^*_{i+1}.
 \end{align*}
 Therefore, there exists $\gamma^* \in \mathbb F$ such that
 \begin{align*}
 \theta^*_{i-1} - \beta \theta^*_i + \theta^*_{i+1} = \gamma^* \qquad \qquad (i \in \mathbb Z).
 \end{align*}
 For $i \in \mathbb Z$ we have
 \begin{align*}
 0 &= (\theta^*_{i-1}- \theta^*_{i+1})(\theta^*_{i-1} - \beta \theta^*_i + \theta^*_{i+1} - \gamma^*) \\
 &=   \theta^{*2}_{i-1}- \beta \theta^*_{i-1} \theta^*_i + \theta^{*2}_i - \gamma^*(\theta^*_{i-1} + \theta^*_i) \\
      &\qquad \qquad -\Bigl(   \theta^{*2}_{i}- \beta \theta^*_{i} \theta^*_{i+1} + \theta^{*2}_{i+1} - \gamma^*(\theta^*_{i} + \theta^*_{i+1})\Bigr).
 \end{align*}
 Therefore,
 there exists $\varrho^* \in \mathbb F$ such that
 \begin{align*}
 \theta^{*2}_{i-1}- \beta \theta^*_{i-1} \theta^*_i + \theta^{*2}_i - \gamma^*(\theta^*_{i-1} + \theta^*_i) = \varrho^* \qquad \qquad (i \in \mathbb Z).
 \end{align*}
 Define a polynomial $P \in \mathbb F\lbrack \lambda, \mu \rbrack$ by
 \begin{align*}
 P(\lambda, \mu) = \lambda^2 - \beta \lambda \mu  + \mu^2- \gamma^*(\lambda+ \mu) -\varrho^*.
 \end{align*}
 Note that $P(\lambda,\mu)= P(\mu, \lambda)$. By construction,
 \begin{align*}
 P(\theta^*_{i-1}, \theta^*_i)=0 \qquad \qquad (i \in \mathbb Z).
 \end{align*}
 We show
 \begin{align}
 0 &= \lbrack A^*, A^{*2} A-\beta A^* A A^* + A A^{*2} - \gamma^*(A^*A+A A^*) -\varrho^* A \rbrack.  \label{eq:TD2again}
 \end{align}
Let $C$ denote the right-hand side of \eqref{eq:TD2again}. We show $C=0$. We have
\begin{align*}
C = \sum_{i=0}^d \sum_{j=0}^d E^*_i C E^*_j.
\end{align*}
For $0 \leq i,j\leq d$ we evaluate  $E^*_i C E^*_j$ using $E^*_i A^*= \theta^*_i E^*_i$ and $A^* E^*_j = \theta^*_j E^*_j$.
This yields
\begin{align*}
E^*_i C E^*_j &= E^*_i A E^*_j (\theta^*_i - \theta^*_j)(\theta_i^{*2} - \beta \theta^*_i \theta^*_j + \theta_j^{*2} - \gamma^* (\theta^*_i + \theta^*_j) -\varrho^*) \\
& =E^*_i A E^*_j (\theta^*_i - \theta^*_j)P(\theta^*_i, \theta^*_j).
\end{align*} 
If $1 < \vert i-j\vert <d$ then $E^*_i A E^*_j=0$.
If $\vert i-j \vert \in \lbrace 1,d\rbrace$  then $P(\theta^*_i, \theta^*_j)=0$.
If $i=j$ then of course $\theta^*_i - \theta^*_j=0$.
In any case, $E^*_i C E^*_j=0$. By these comments, $C=0$. We have established \eqref{eq:TD2again}. 
The scalars $\beta, \gamma^*, \varrho^*$ satisfy Proposition \ref{prop:sum}(iii). 
This contradicts our assumption that the
equivalent conditions in Proposition \ref{prop:sum} do not hold for $\beta$. Therefore, $\xi \not=0$.
\end{proof}

\begin{lemma} \label{lem:dodd} 
The integer $d$ is odd.
\end{lemma}
\begin{proof} By Lemmas \ref{lem:checksum}, \ref{lem:check3}, \ref{lem:check4}  we have
\begin{align*}
 \sum_{i=0}^d (-1)^i = \xi^{-1} \sum_{i=0}^d (\theta^*_i)^\vee =0.
\end{align*}
Consequently, $d$ is odd.
\end{proof}

\noindent Let $\overline{\mathbb F}$ denote the algebraic closure of $\mathbb F$.

\begin{lemma} \label{lem:lem7} We express the scalars $\lbrace \theta^*_i \rbrace_{i=0}^d$ in closed form.
\begin{enumerate}
\item[\rm (i)] 
Assume that 
${\rm Char}(\mathbb F) \not=2$ and $\beta \not\in \lbrace 2,-2\rbrace$.
  Write $\beta = q + q^{-1}$ with $q \in \overline{\mathbb F}$.
Then 
\begin{align}
 \theta^*_i = \alpha^*_0 + \alpha^*_1 q^i + \alpha^*_2 q^{-i} + \frac{1}{2} \,\frac{ \xi}{\beta+2}  (-1)^i \qquad \qquad (0 \leq i \leq d). \label{eq:thForm}
\end{align}
Moreover, $q^{d+1}=1$. 
\item[\rm (ii)] Assume that  ${\rm Char}(\mathbb F) \not=2$ and $\beta =2$.
Then 
\begin{align}
\theta^*_i = \alpha^*_0 + \alpha^*_1 i + \alpha^*_2 i^2+ \frac{\xi}{8} (-1)^i \qquad \qquad (0 \leq i \leq d). \label{eq:thForm2}
\end{align}
Moreover, ${\rm Char}(\mathbb F)$ divides $d+1$. 
\item[\rm (iii)] Assume that  ${\rm Char}(\mathbb F) \not=2$ and $\beta =-2$.
Then 
\begin{align}
 \theta^*_i = \alpha^*_0 + \alpha^*_1 (-1)^i + \alpha^*_2 (-1)^i i + \frac{\xi}{4} (-1)^i i^2\qquad \qquad (0 \leq i \leq d). \label{eq:thFormm2}
\end{align}
Moreover, ${\rm Char}(\mathbb F)$ divides $d+1$. 
\item[\rm (iv)] Assume that ${\rm Char}(\mathbb F) =2$
and $\beta \not=2$. 
 Write $\beta = q + q^{-1}$ with $q \in \overline{\mathbb F}$.
Then 
\begin{align}
\theta^*_i = \alpha^*_0 + \alpha^*_1 q^i + \alpha^*_2 q^{-i} +\frac{\xi}{\beta-2} i\qquad \qquad (0 \leq i \leq d). \label{eq:thFormc}
\end{align}
Moreover, $q^{d+1}=1$. 
\item[\rm (v)] Assume that ${\rm Char}(\mathbb F)=2$ and $\beta=2$. Then $\Phi$ does not exist.
\end{enumerate}
\end{lemma}
\begin{proof} (i) For the linear recurrence in Lemma \ref{lem:lem6a}, the characteristic polynomial is
\begin{align*}
 \lambda^4-1-\beta (\lambda^3-\lambda) = (\lambda-1)(\lambda+1)(\lambda-q)(\lambda-q^{-1}).
 \end{align*}
 For this polynomial the roots $1,-1,q,q^{-1}$ are mutually distinct. Therefore, the linear recurrence has general solution
 \begin{align*}
  \theta^*_i = \alpha^*_0 + \alpha^*_1 q^i + \alpha^*_2 q^{-i} + \alpha^*_3 (-1)^i \qquad \qquad (i \in \mathbb Z). 
 \end{align*}
 By this formula along with Definition \ref{def:check1} and Lemma \ref{lem:check3},
 \begin{align*}
 \xi = \frac{(\theta^*_i)^\vee}{(-1)^i} = \frac{\theta^*_{i-2} - \theta^*_{i+1} - (\beta+1)(\theta^*_{i-1}-\theta^*_i)}{(-1)^i} = 2(\beta+2) \alpha^*_3
 \end{align*} 
 for $i \in \mathbb Z$.
 By these comments, we obtain \eqref{eq:thForm}. The remaining assertion follows
 from  \eqref{eq:mod}.
 \\
 \noindent (ii)--(iv) Similar. \\
 \noindent (v) We have $\beta = 2 =0$, which contradicts Lemma \ref{lem:betnz}. 
 \end{proof}
\noindent In Lemma  \ref{lem:lem5} we obtained a linear dependency among five polynomials in $A, A^*$. This linear dependency was used in Lemma \ref{lem:lem6a} to get a linear recurrence involving
the eigenvalues $\lbrace \theta^*_i \rbrace_{i=0}^d$. Our next general goal is to use this linear dependency to get some information about the eigenvalues $\lbrace \theta_i \rbrace_{i=0}^d$.

\begin{definition} \label{alt:bbT}
\rm Define the $(d+1) \times 5 $  matrix     
\begin{align*}
{\mathbb T} =\begin{pmatrix}
          \theta_0-\theta_1 & \theta_0^2-\theta_1^2& \theta_0^3-\theta_1^3& \theta^2_0 \theta_1 -\theta_0\theta_1^2&  \theta_0^4 - \theta_1^4-\beta(\theta^3_0 \theta_1 -\theta_0 \theta_1^3) \\
           \theta_1-\theta_2 & \theta_1^2-\theta_2^2& \theta_1^3-\theta_2^3& \theta_1^2 \theta_2 -\theta_1 \theta_2^2&     \theta_1^4 - \theta_2^4     -\beta    ( \theta^3_1 \theta_2-\theta_1 \theta_2^3) \\
            \theta_2-\theta_3 & \theta_2^2-\theta_3^2& \theta_2^3-\theta_3^3& \theta_2^2 \theta_3-\theta_2 \theta_3^2 &  \theta_2^4 - \theta_3^4  -\beta ( \theta_2^3 \theta_3-\theta_2 \theta_3^3) \\
            \vdots &\vdots&\vdots&\vdots&\vdots \\
             \theta_{d-1}-\theta_d & \theta_{d-1}^2-\theta_d^2& \theta_{d-1}^3-\theta_d^3& \theta_{d-1}^2 \theta_d-\theta_{d-1} \theta_d^2&   \theta_{d-1}^4 - \theta_d^4   -\beta (      \theta_{d-1}^3 \theta_d-\theta_{d-1}\theta_d^3)  \\
              \theta_d-\theta_0 & \theta_d^2-\theta_0^2& \theta_d^3-\theta_0^3& \theta_d^2 \theta_0-\theta_d\theta_0^2 &   \theta_d^4 - \theta_0^4      -\beta (      \theta_d^3 \theta_0-\theta_d \theta_0^3 )
          \end{pmatrix}.
\end{align*}
\end{definition}

\begin{lemma} \label{alt:dep} For the matrix $\mathbb T$, the rightmost column is contained in the span of the first four columns.
\end{lemma}
\begin{proof} By Lemma \ref{lem:lem5}, there exists $ \gamma_\ell \in  \mathbb F$ $(1 \leq \ell \leq 4)$ such that
\begin{align*}
&A^4 A^*-A^* A^4 -\beta (A^3 A^* A- A A^* A^3) \\
&\qquad = \gamma_1 (A A^*-A^*A) + \gamma_2 (A^2 A^*-A^*A^2) + \gamma_3 (A^3 A^*-A^*A^3) + \gamma_4 (A^2 A^* A-A A^* A^2).
\end{align*}
We will show that $\mathbb T \psi =0$, where $\psi = (\gamma_1, \gamma_2, \gamma_3, \gamma_4, -1)^{\rm t}$.
In the previously displayed equation, let $Z$ denote the left-hand side minus the right-hand side. We have $Z=0$, so
\begin{align*}
 E_i Z E_j =0 \qquad \qquad (0 \leq i,j\leq d).
 \end{align*}
For $0 \leq i,j\leq d$ we evaluate $E_i Z E_j$ using $E_i A = \theta_i E_i$ and $A E_j = \theta_j E_j$; this yields
\begin{align}
0 = E_i Z E_j = E_i A^* E_j c_{i,j} \label{eq:triplep}
\end{align}
where
\begin{align*}
c_{i,j}  &=  \theta_i^4-\theta_j^4 
-\beta (\theta^3_i \theta_j -\theta_i \theta^3_j)
-
\gamma_1 (\theta_i - \theta_j) - \gamma_2 (\theta_i^2-\theta_j^2) 
\\
& \qquad \qquad - \gamma_3 (\theta_i^3-\theta_j^3)
  -\gamma_4 (\theta^2_i \theta_j -\theta_i \theta_j^2).
\end{align*}
Since the Hessenberg system $\Phi$ is circular, we have  $E_i A^* E_{i-1} \not=0$  $(1 \leq i \leq d)$ and $E_0 A^* E_d \not=0$. This and \eqref{eq:triplep} yield
$c_{i,i-1} =0$  $(1 \leq i \leq d)$ and $c_{0,d}=0$.
By Definition \ref{alt:bbT} and matrix multiplication,
\begin{align*}
\mathbb T \psi = \bigl( c_{1,0}, c_{2,1}, \ldots, c_{d,d-1}, c_{0,d} \bigr)^{\rm t} = 0.
\end{align*}
The result follows.
\end{proof}

\begin{lemma} \label{alt:dep2} The rank of $\mathbb T$ is at most 4.
\end{lemma}
\begin{proof} By Lemma \ref{alt:dep}.
\end{proof}

\begin{definition}\label{alt:Ti} \rm
For $0 \leq i \leq d-4$, let $\mathbb T_i$ denote the $5 \times 5$ matrix obtained from $\mathbb T$ by removing the top $i$ rows and the bottom $d-4-i$ rows.
\end{definition}

\begin{lemma} \label{alt:detTi}
For $0 \leq i \leq d-4$ we have ${\rm det}(\mathbb T_i)=0$.
\end{lemma}
\begin{proof} By Lemma \ref{alt:dep2} and linear algebra.
\end{proof}

\noindent Next, we use Lemma \ref{alt:detTi} to solve for $\beta$.

\begin{definition} \label{alt:bbTp}
\rm Define the $(d+1) \times 5 $  matrices     
\begin{align*}
{}^+{\mathbb T} =\begin{pmatrix}
          \theta_0-\theta_1 & \theta_0^2-\theta_1^2& \theta_0^3-\theta_1^3& \theta^2_0 \theta_1 -\theta_0\theta_1^2&  \theta_0^4 - \theta_1^4 \\
           \theta_1-\theta_2 & \theta_1^2-\theta_2^2& \theta_1^3-\theta_2^3& \theta_1^2 \theta_2 -\theta_1 \theta_2^2&     \theta_1^4 - \theta_2^4    \\
            \theta_2-\theta_3 & \theta_2^2-\theta_3^2& \theta_2^3-\theta_3^3& \theta_2^2 \theta_3-\theta_2 \theta_3^2 &  \theta_2^4 - \theta_3^4   \\
            \vdots &\vdots&\vdots&\vdots&\vdots \\
             \theta_{d-1}-\theta_d & \theta_{d-1}^2-\theta_d^2& \theta_{d-1}^3-\theta_d^3& \theta_{d-1}^2 \theta_d-\theta_{d-1} \theta_d^2&   \theta_{d-1}^4 - \theta_d^4   \\
              \theta_d-\theta_0 & \theta_d^2-\theta_0^2& \theta_d^3-\theta_0^3& \theta_d^2 \theta_0-\theta_d\theta_0^2 &   \theta_d^4 - \theta_0^4 
          \end{pmatrix}
\end{align*}
and
\begin{align*}
{}^-{\mathbb T} =\begin{pmatrix}
          \theta_0-\theta_1 & \theta_0^2-\theta_1^2& \theta_0^3-\theta_1^3& \theta^2_0 \theta_1 -\theta_0\theta_1^2& \theta^3_0 \theta_1 -\theta_0 \theta_1^3 \\
           \theta_1-\theta_2 & \theta_1^2-\theta_2^2& \theta_1^3-\theta_2^3& \theta_1^2 \theta_2 -\theta_1 \theta_2^2&   \theta^3_1 \theta_2-\theta_1 \theta_2^3 \\
            \theta_2-\theta_3 & \theta_2^2-\theta_3^2& \theta_2^3-\theta_3^3& \theta_2^2 \theta_3-\theta_2 \theta_3^2 &   \theta_2^3 \theta_3-\theta_2 \theta_3^3 \\
            \vdots &\vdots&\vdots&\vdots&\vdots \\
             \theta_{d-1}-\theta_d & \theta_{d-1}^2-\theta_d^2& \theta_{d-1}^3-\theta_d^3& \theta_{d-1}^2 \theta_d-\theta_{d-1} \theta_d^2&    \theta_{d-1}^3 \theta_d-\theta_{d-1}\theta_d^3  \\
              \theta_d-\theta_0 & \theta_d^2-\theta_0^2& \theta_d^3-\theta_0^3& \theta_d^2 \theta_0-\theta_d\theta_0^2 &   \theta_d^3 \theta_0-\theta_d \theta_0^3 
          \end{pmatrix}.
\end{align*}
\end{definition}

\begin{definition}\label{alt:Tipm} \rm
For $0 \leq i \leq d-4$, let ${}^+\mathbb T_i$ (resp.  ${}^-\mathbb T_i$)     denote the $5 \times 5$ matrix obtained from ${}^+\mathbb T$ (resp. ${}^-\mathbb T$ )  by removing the top $i$ rows and the bottom $d-4-i$ rows.
\end{definition}

\begin{lemma} \label{alt:solbeta}
For $0 \leq i \leq d-4$, 
\begin{align*}
{\rm det}({}^+ \mathbb T_i) = \beta {\rm det}({}^- \mathbb T_i).
\end{align*}
\end{lemma}
\begin{proof} By Lemma \ref{alt:detTi} and the construction,
\begin{align*}
0 = {\rm det}(\mathbb T_i) = {\rm det}({}^+\mathbb T_i) - \beta {\rm det}({}^-\mathbb T_i).
\end{align*}
\end{proof}

\noindent Applying Lemma \ref{alt:solbeta} to $\Phi^*$, we find that
for $0 \leq i \leq d-4$, 
\begin{align}
{\rm det}({}^+ \mathbb T^*_i) = \beta^* {\rm det}({}^- \mathbb T^*_i). \label{eq:maine}
\end{align}

\begin{lemma} \label{alt:detMt} For $0 \leq i \leq d-4$,
\begin{align*}
{\rm det}({}^- \mathbb T^*_i) &= 
 (\theta^*_i-\theta^*_{i+1})(\theta^*_{i+1}-\theta^*_{i+2}) (\theta^*_{i+2}-\theta^*_{i+3})(\theta^*_{i+3}-\theta^*_{i+4})(\theta^*_{i+4}-\theta^*_{i+5}) \\
&\quad  \times (\theta^*_i-\theta^*_{i+2})(\theta^*_{i+1}-\theta^*_{i+3})^2 (\theta^*_{i+2} -\theta^*_{i+4})^2 (\theta^*_{i+3}-\theta^*_{i+5})  \\
& \quad \times (\theta^*_{i+1}-\theta^*_{i+4}) \beta \xi (-1)^i
\end{align*}
where we understand $\theta^*_{d+1}=\theta^*_0$.  Moreover, ${\rm det}({}^- \mathbb T^*_i) \not=0$. 
\end{lemma}
\begin{proof} The first assertion is checked using the data in Lemma \ref{lem:lem7}. The last assertion follows from the first assertion along with
Lemmas \ref{lem:betnz}, \ref{lem:check4}. 
\end{proof}

\begin{lemma} \label{alt:detPt} For $0 \leq i \leq d-4$,
\begin{align*}
{\rm det}({}^+ \mathbb T^*_i) &= 
 (\theta^*_i-\theta^*_{i+1})(\theta^*_{i+1}-\theta^*_{i+2}) (\theta^*_{i+2}-\theta^*_{i+3})(\theta^*_{i+3}-\theta^*_{i+4})(\theta^*_{i+4}-\theta^*_{i+5}) \\
&\quad  \times (\theta^*_i-\theta^*_{i+2})(\theta^*_{i+1}-\theta^*_{i+3})^2 (\theta^*_{i+2} -\theta^*_{i+4})^2 (\theta^*_{i+3}-\theta^*_{i+5})  \\
& \quad \times   \bigl(\xi +                      (-1)^i  \beta (\theta^*_{i+1}-\theta^*_{i+4}) \bigr) \beta \xi
\end{align*}
where we understand $\theta^*_{d+1}=\theta^*_0$.  
\end{lemma}
\begin{proof} This is checked using the data in Lemma \ref{lem:lem7}. 
\end{proof}

\begin{lemma} \label{alt:bb} For $0 \leq i \leq d-4$, 
\begin{align*}
\beta^* = \beta + \frac{\xi (-1)^i}{\theta^*_{i+1}-\theta^*_{i+4}}. 
\end{align*}
\end{lemma}
\begin{proof} Evaluate \eqref{eq:maine} using Lemmas \ref{alt:detMt}, \ref{alt:detPt} and simplify the result.
\end{proof}

\begin{lemma} \label{alt:qed} The circular Hessenberg system $\Phi$ does not exist.
\end{lemma}
\begin{proof} Note that $d\geq 5$ since $d$ is odd. 
By Lemma \ref{alt:bb} and $\xi \not=0$, the scalar $(-1)^i (\theta^*_{i+1} - \theta^*_{i+4})$ is independent of $i$ for $0 \leq i \leq d-4$.
Taking $i \in \lbrace 0,1\rbrace$ we obtain
\begin{align*}
\theta^*_1 - \theta^*_4 = - (\theta^*_2-\theta^*_5).
\end{align*}
By this and Lemma \ref{lem:lem6a},
\begin{align*}
0 &= \theta^*_1+\theta^*_2-\theta^*_4-\theta^*_5 
 = (\beta+1)(\theta^*_2 - \theta^*_4).
\end{align*}
Therefore $\beta=-1$. Applying this result to $\Phi^*$, we obtain $\beta^*=-1$. Now $\beta=\beta^*$, which contradicts Lemma \ref{alt:bb} and $\xi \not=0$. We conclude that
$\Phi$ does not exist.
\end{proof}

\noindent  In this section, we proved Conjecture \ref{conj} for  $d\geq 4$. In previous sections, we proved Conjecture \ref{conj} for $d=2$ and $d=3$.
 We have proven Conjecture \ref{conj}. 



\bigskip

\noindent Kazumasa Nomura \hfil\break
\noindent Institute of Science Tokyo \hfil\break
\noindent Kohnodai Ichikawa 272-0827 Japan \hfil\break
\noindent email: {\tt knomura@pop11.odn.ne.jp} \hfil\break

\noindent Paul Terwilliger \hfil\break
\noindent Department of Mathematics \hfil\break
\noindent University of Wisconsin \hfil\break
\noindent 480 Lincoln Drive \hfil\break
\noindent Madison, WI 53706-1388 USA \hfil\break
\noindent email: {\tt terwilli@math.wisc.edu }\hfil\break

\section{Statements and Declarations}

\noindent {\bf Funding}: The author declares that no funds, grants, or other support were received during the preparation of this manuscript.
\medskip

\noindent  {\bf Competing interests}:  The author  has no relevant financial or non-financial interests to disclose.
\medskip

\noindent {\bf Data availability}: All data generated or analyzed during this study are included in this published article.

\end{document}